\documentclass[11pt,reqno]{amsart}

\usepackage{amsmath,amssymb,amsthm}
\usepackage{mathtools}
\usepackage{mathrsfs}
\usepackage[a4paper,textwidth=16.3cm,textheight=23.4cm,centering]{geometry}
\usepackage{enumitem}
\usepackage{needspace}
\usepackage[expansion=false]{microtype}

\renewcommand\baselinestretch{1.1}
\AtBeginDocument{%
  \setlength{\abovedisplayskip}{4pt plus 2pt minus 2pt}%
  \setlength{\belowdisplayskip}{4pt plus 2pt minus 2pt}%
  \setlength{\abovedisplayshortskip}{3.5pt plus 2pt minus 2pt}%
  \setlength{\belowdisplayshortskip}{3.5pt plus 2pt minus 2pt}%
  \setlength{\jot}{3pt}%
}
\setlist{itemsep=2pt,topsep=4pt,parsep=0pt,partopsep=0pt}

\numberwithin{equation}{section}
\theoremstyle{plain}
\newtheorem{theorem}{Theorem}[section]
\newtheorem{proposition}[theorem]{Proposition}
\newtheorem{lemma}[theorem]{Lemma}
\newtheorem{corollary}[theorem]{Corollary}
\newtheorem*{theorema}{Theorem A}
\newtheorem*{theoremb}{Theorem B}
\theoremstyle{definition}
\newtheorem{definition}[theorem]{Definition}

\newtheorem{remark}[theorem]{Remark}

\usepackage[colorlinks=true,linkcolor=blue,citecolor=blue,urlcolor=blue]{hyperref}
\newcommand{\kk}{\Bbbk}
\newcommand{\OO}{\mathcal O}
\newcommand{\Aut}{\operatorname{Aut}}
\newcommand{\Id}{\operatorname{Id}}
\newcommand{\End}{\operatorname{End}}
\newcommand{\Der}{\operatorname{Der}}
\newcommand{\Lop}{L}
\newcommand{\Rop}{R}

\title[Classification of Simple Novikov Algebras]
{Classification of Finite-Dimensional Simple\\
Novikov Algebras and Their Irreducible Bimodules\\
in Positive Characteristic}
\author{Dong Liu}
\address{Department of Mathematics, Huzhou Normal University,
Huzhou 313000, China}
\email{liudong@huznu.edu.cn}
\author{Yufeng Pei}
\address{Department of Mathematics, Huzhou Normal University,
Huzhou 313000, China}
\email{pei@huznu.edu.cn}
\author{Limeng Xia}
\address{College of Mathematical Sciences, Jiangsu University,
Zhenjiang 212013, China}
\email{xialimeng@ujs.edu.cn}
\hypersetup{
  pdftitle={Classification of Finite-Dimensional Simple Novikov Algebras and Their Irreducible Bimodules in Positive Characteristic},
  pdfauthor={Dong Liu, Yufeng Pei, Limeng Xia},
  pdfsubject={Classification of simple Novikov algebras and irreducible bimodules},
  pdfkeywords={Novikov algebra, simple algebra, irreducible bimodule, positive characteristic}
}
\date{}

\begin{document}

\begin{abstract}
We give a complete classification of finite-dimensional simple
Novikov algebras over an algebraically closed field of characteristic
$p>2$, including families missing from earlier classifications in \cite{Osborn1992b, Xu1996}.
We determine explicit normal forms and all isomorphisms between them,
and classify all nontrivial irreducible bimodules over these algebras.
\end{abstract}

\maketitle

\section{Introduction}
\label{sec:introduction}
\raggedbottom

Novikov algebras were introduced by Gelfand and Dorfman in 1979
in the study of Hamiltonian operators~\cite{GD1979}, and independently
by Balinsky and Novikov in 1985 in connection with Poisson brackets of
hydrodynamic type~\cite{BN1985}.
Osborn~\cite{Osborn1992a,Osborn1992b} introduced the term
\emph{Novikov algebra} and studied the structure and classification of
finite-dimensional simple Novikov algebras.
In characteristic zero, Novikov algebras give rise to infinite-dimensional
Lie algebras by affinization, and to Lie conformal algebras and vertex
algebras~\cite{BN1985,BLP2015}.
Novikov bialgebras similarly yield infinite-dimensional Lie bialgebras
and Lie conformal bialgebras~\cite{HBG2023,HBG2025}.
A Novikov algebra is solvable if and only if it
is right nilpotent~\cite[Theorem~3.3]{SZ2020}.
Recent results concern polynomial identities~\cite{DIU2023} and the
classification of finite simple Novikov conformal algebras~\cite{KL2026}.

Zelmanov~\cite{Zelmanov1987} proved that every finite-dimensional simple
Novikov algebra over an algebraically closed field of characteristic
zero is one-dimensional.
Every finite-dimensional simple Novikov superalgebra over a field of
characteristic different from $2$ has zero odd part~\cite{LPX2019}.
In positive characteristic, Osborn proposed a classification of
finite-dimensional simple Novikov algebras containing a nonzero
idempotent~\cite{Osborn1992b}.
Xu proposed the same classification without the assumption of a nonzero
idempotent~\cite[Main Theorem~1]{Xu1996}
and classified the finite-dimensional irreducible bimodules over the
algebras in his classification~\cite[Main Theorem~2]{Xu1996}
(see also~\cite{Osborn1995modules}).
The classification of Osborn and Xu is incomplete in dimensions $p^n$
with $n\ge2$; Remark~\ref{rem:Osborn-reduction} gives a $p^2$-dimensional
simple Novikov algebra outside their family.

For a commutative associative algebra $\mathcal{A}$, a derivation
$\partial\in\Der(\mathcal{A})$, and $a\in\mathcal{A}$, the twisted
Gelfand--Dorfman construction gives the Novikov algebra
$\mathcal{A}(\partial,a)$ with product
\[
 x\circ y=x\bigl(\partial(y)+ay\bigr).
\]
This product is obtained by adding $axy$ to the differential product
$x\partial(y)$ in~\cite{GD1979}.
The scalar twist was introduced in~\cite{Filippov1989}, and the general
twist $a\in\mathcal{A}$ in~\cite{Xu1997,Xu2001}.
A related PBW embedding theorem for Novikov--Poisson algebras with a
unital associative product was proved in~\cite[Theorem~3.3]{BCZ2018}.

By the structure theorem of Zhelyabin and
Zakharov~\cite[Theorem~3 and Corollary~1]{ZZ2024}, every finite-dimensional
simple Novikov algebra over an algebraically closed field of
characteristic $p>2$ is obtained by the twisted Gelfand--Dorfman
construction, with $\mathcal{A}$ a $\partial$-simple truncated polynomial
algebra.  The one-dimensional case is included by taking
$\partial=0$ and $a=1$ on the ground field.
The corresponding existence theorem over arbitrary fields of positive
characteristic was proved in~\cite[Theorem~2.4]{PZ2026}, where
automorphisms were also studied.
An isomorphism criterion for these algebras was subsequently obtained
in~\cite{PZZ2025}.
The remaining problem is to find explicit normal forms for these
algebras and determine when they are isomorphic.
The aim of this paper is to complete this classification.

\Needspace{3\baselineskip}
In what follows, $\kk$ is an algebraically closed field of characteristic
$p>2$, and $\mathbb F_p$ denotes its prime subfield.
For $n\ge0$, let $\OO(1;n)$ be the truncated divided-power algebra with basis
$x^{(0)},\ldots,x^{(p^n-1)}$ and lowering derivation $\delta$.
The definitions and conventions are collected in
Section~\ref{subsec:divided-power-algebras}; in particular,
$\OO(1;0)=\kk$ and $\delta=0$ on $\OO(1;0)$.
For $\alpha=(\alpha_1,\ldots,\alpha_n)\in\kk^n$ and
$\beta\in\kk$, with $\beta=1$ if $n=0$, put
\begin{equation}
 b_{\alpha,\beta}
 =\sum_{j=1}^n\alpha_jx^{(p^j-2)}+\beta x^{(p^n-1)},
 \qquad
 \mathcal{N}(\alpha,\beta)=\OO(1;n)(\delta,b_{\alpha,\beta}).
 \label{eq:normal-form-parameters}
\end{equation}
When $n=0$, the tuple $\alpha$ is empty and
$\mathcal{N}(\alpha,1)=\kk$ with its usual product.
In the basis $y_i=x^{(i+1)}$, $-1\le i\le p^n-2$, the product of
$\mathcal{N}(\alpha,\beta)$ is
\begin{equation}
\begin{split}
 y_i\circ y_j
 &=\binom{i+j+1}{j}y_{i+j}
   +\beta\delta_{i,-1}\delta_{j,-1}y_{p^n-2}\\
 &\quad+\sum_{s=1}^n\alpha_s
   \binom{i+j+2}{i+1}\binom{i+j+p^s}{p^s-2}
   y_{i+j+p^s-1}.
\end{split}
\label{eq:explicit-algebra-product}
\end{equation}
Here $-1\le i,j\le p^n-2$, $y_r=0$ for $r\notin\{-1,\ldots,p^n-2\}$,
and $\delta_{ij}$ is the Kronecker delta.
By Lemma~\ref{lem:normalized-simple}, each algebra
$\mathcal{N}(\alpha,\beta)$ is simple of dimension $p^n$.

\begin{theorema}[Theorem \ref{thm:final-classification} below]
\begin{enumerate}[label=\textup{(\roman*)},leftmargin=2em]
\item
Every finite-dimensional simple Novikov algebra is isomorphic
to $\mathcal{N}(\alpha,\beta)$ for some $n\ge0$ and parameters as above.

\item
Two algebras
$\mathcal{N}(\alpha,\beta)$ and $\mathcal{N}(\alpha',\beta')$
of dimension $p^n$ are isomorphic if and only if there is
$c\in\kk^\times$ such that
\[
 \alpha'_j=c^{p^j-1}\alpha_j\quad(1\le j\le n),
 \qquad \beta'=c^{p^n}\beta.
\]
\end{enumerate}
\end{theorema}

The proof starts from the structure theorem
in~\cite[Theorem~3 and Corollary~1]{ZZ2024} and the isomorphism criterion
in~\cite[Theorem~3.3]{PZZ2025}.
Using an argument from~\cite{Premet2015}, we reduce $\partial$ to
$\delta$ by coordinate changes and multiplication by units.
Successive coefficient elimination by admissible automorphisms with
linear coefficient $1$ gives a unique normal form $b_{\alpha,\beta}$.
Scalings then give the parameter identifications in Theorem~A.
As a consequence, we classify finite-dimensional Novikov--Poisson
algebras whose Novikov part is simple
(Corollary~\ref{cor:novikov-poisson-classification}).

For the definition and basic properties of Novikov bimodules,
see~\cite{Osborn1995modules,Xu1996,Xu2001}.
A bimodule is called nontrivial if its two actions are not both zero.
For each $\mathcal{N}(\alpha,\beta)$ and $\lambda\in\kk$, let
$\mathcal{M}(\alpha,\beta,\lambda)$ be the $\mathcal{N}(\alpha,\beta)$-bimodule
with basis $v_{-1},\ldots,v_{p^n-2}$ and actions
\begin{align}
 y_i\circ v_j
 &=\binom{i+j+1}{j}v_{i+j}
   +\lambda\binom{i+j+2}{i+1}v_{i+j+1},
 \label{eq:explicit-module-left}\\
 v_i\circ y_j
 &=\binom{i+j+1}{j}v_{i+j}
   +\beta\delta_{i,-1}\delta_{j,-1}v_{p^n-2}\notag\\
 &\quad+\sum_{s=1}^n\alpha_s
   \binom{i+j+2}{i+1}\binom{i+j+p^s}{p^s-2}
   v_{i+j+p^s-1}.
 \label{eq:explicit-module-right}
\end{align}
Here $-1\le i,j\le p^n-2$, and $v_r=0$ for
$r\notin\{-1,\ldots,p^n-2\}$.
For $n=0$, these formulas give $1\circ v_{-1}=\lambda v_{-1}$ and
$v_{-1}\circ1=v_{-1}$.
The parameter $\lambda$ is the unique eigenvalue of the left action of $1$.

\Needspace{6\baselineskip}
\begin{theoremb}[Theorem \ref{thm:irreducible-bimodule-classification} below]
Fix $\mathcal{N}(\alpha,\beta)$ as in Theorem~A.
\begin{enumerate}[label=\textup{(\roman*)},leftmargin=2em]
\item
For each $\lambda\in\kk$, the bimodule
$\mathcal{M}(\alpha,\beta,\lambda)$ is nontrivial and irreducible
of dimension $p^n$.  Every nontrivial irreducible
$\mathcal{N}(\alpha,\beta)$-bimodule is isomorphic to one of these.

\item
The $\mathcal{N}(\alpha,\beta)$-bimodules
$\mathcal{M}(\alpha,\beta,\lambda)$ and
$\mathcal{M}(\alpha,\beta,\lambda')$ are isomorphic if and only if
$\lambda=\lambda'$.
\end{enumerate}
\end{theoremb}

For Theorem~B, we use polynomial relations among right multiplication
operators to recover an action of $\OO(1;n)$ on every nontrivial
irreducible bimodule, following~\cite{Osborn1995modules,Xu1996,Xu2001}.
After proving finite-dimensionality, we show that an eigenvector of
the left action of $1$ generates the bimodule freely over $\OO(1;n)$.
This gives $\mathcal{M}(\alpha,\beta,\lambda)$.

Section~\ref{sec:changing-data} recalls the twisted Gelfand--Dorfman
construction and the isomorphism criterion.
Section~\ref{sec:derivation-normalization} normalizes the derivation.
Sections~\ref{sec:residual-orbit} and~\ref{sec:bimodules} prove
Theorems~A and~B, respectively.

\flushbottom
\section{Twisted Gelfand--Dorfman construction and isomorphisms}
\label{sec:changing-data}

We recall divided-power algebras and the twisted Gelfand--Dorfman
construction, then state the structure and isomorphism theorems.

\subsection{Divided-power algebras}
\label{subsec:divided-power-algebras}

We use $\OO(1)$ and $\OO(1;n)$ for the one-variable divided-power
algebras in~\cite[\S3.2]{PremetStrade2006}.
We write $f^{(r)}$ for divided powers and $f^r$ for ordinary powers.
Binomial coefficients are computed in $\mathbb Z$ and then viewed as
elements of $\kk$, with $\binom ab=0$ unless $0\le b\le a$.

\Needspace{11\baselineskip}
\begin{definition}
The \emph{divided-power algebra} in one variable is
\[
 \OO(1)=\bigoplus_{r\ge0}\kk x^{(r)},
 \qquad x^{(0)}=1,\quad x=x^{(1)},
\]
with multiplication
\begin{equation}
 x^{(r)}x^{(s)}=\binom{r+s}{r}x^{(r+s)}
 \qquad(r,s\ge0).
 \label{eq:divided-power-product}
\end{equation}
For $n\ge1$, the \emph{truncated divided-power algebra in one variable
of height $n$} is
\begin{equation}
 \OO(1;n)=\bigoplus_{0\le r<p^n}\kk x^{(r)}.
 \label{eq:finite-divided-power-algebra}
\end{equation}
By convention, $\OO(1;0)=\kk$.
\end{definition}

The product is commutative and associative; associativity follows from
the integer identity
\[
 \binom{r+s}{r}\binom{r+s+t}{r+s}
 =\binom{s+t}{s}\binom{r+s+t}{r}.
\]
To see that~\eqref{eq:finite-divided-power-algebra} is a subalgebra,
recall Lucas's congruence
\begin{equation}
 \binom ab=\prod_{i\ge0}\binom{a_i}{b_i}
 \qquad\text{in }\kk,
 \label{eq:Lucas-congruence}
\end{equation}
where $a_i,b_i$ are the base-$p$ digits of $a,b\ge0$.
If $r,s<p^n\le r+s$, addition in base $p$ has a carry, so
$\binom{r+s}{r}=0$ in $\kk$.  Thus $\OO(1;n)$ is a unital
subalgebra of $\OO(1)$ of dimension $p^n$, and products of its basis
elements vanish when their indices add to at least $p^n$.
We set $x^{(t)}=0$ for $t<0$.

Following~\cite[\S2.1]{Premet2015}, write
\[
 \OO_n=\kk[X_1,\ldots,X_n]/(X_1^p,\ldots,X_n^p),
 \qquad \OO_0=\kk,
\]
and let $x_i$ denote the residue class of $X_i$.
For $n>0$, Lucas's congruence gives
\begin{equation}
 \bigl(x^{(p^{i-1})}\bigr)^j=j!\,x^{(jp^{i-1})}
 \quad(1\le i\le n,\ 0\le j<p),
 \qquad \bigl(x^{(p^{i-1})}\bigr)^p=0.
 \label{eq:coordinate-powers}
\end{equation}
Hence the assignment
\begin{equation}
 \OO_n\xrightarrow{\ \sim\ }\OO(1;n),
 \qquad x_i\longmapsto x^{(p^{i-1})}
 \label{eq:polynomial-presentation}
\end{equation}
is an isomorphism of commutative associative algebras: it maps the
monomial basis of $\OO_n$ to nonzero scalar multiples of the
divided-power basis of $\OO(1;n)$.
We identify these algebras, so that $x_i=x^{(p^{i-1})}$.
In these polynomial coordinates,
\begin{equation}
 x^{(r)}=\prod_{i=0}^{n-1}\frac{x_{i+1}^{r_i}}{r_i!},
 \qquad r=\sum_{i=0}^{n-1}r_ip^i,
 \qquad 0\le r<p^n,\quad 0\le r_i<p.
 \label{eq:divided-power-digits}
\end{equation}
For $0\le r\le n$, we identify $\OO(1;r)$ with the subalgebra
of $\OO(1;n)$ generated by $x_1,\ldots,x_r$.

For $f=\sum_j a_jx^{(j)}$, write $[x^{(j)}]f=a_j$ and $f(0)=a_0$.
The algebra $\OO(1;n)$ is local with maximal ideal
\[
 \mathfrak m=\{f\in\OO(1;n):f(0)=0\}
 =\operatorname{span}_\kk\{x^{(j)}:1\le j<p^n\}
 =(x_1,\ldots,x_n),
\]
where $\mathfrak m=0$ if $n=0$.
The polynomial presentation and the Frobenius identity give
\begin{equation}
 f^p=f(0)^p1,
 \qquad
 f\in\OO(1;n)^\times\ \Longleftrightarrow\ f(0)\ne0.
 \label{eq:divided-power-units}
\end{equation}
Indeed, if $f(0)=0$, then $f$ is nilpotent; otherwise
$f^{-1}=f(0)^{-p}f^{p-1}$.

The \emph{standard divided-power grading} assigns degree $j$ to
$x^{(j)}$, so $\deg x_i=p^{i-1}$.
For fixed $n$ and every integer $s\ge0$, put
\begin{equation}
 F^s=\operatorname{span}_\kk\{x^{(j)}:s\le j<p^n\}.
 \label{eq:divided-power-filtration}
\end{equation}
These subspaces form the standard divided-power filtration.
We have $F^0=\OO(1;n)$, $F^1=\mathfrak m$, $F^iF^j\subseteq F^{i+j}$
for $i,j\ge0$, and $F^s=0$ for $s\ge p^n$.

The \emph{lowering derivation} on $\OO(1)$ is defined by
\begin{equation}
 \delta(1)=0,\qquad \delta(x^{(r)})=x^{(r-1)}\quad(r\ge1).
 \label{eq:lowering-derivation}
\end{equation}
Pascal's identity proves the Leibniz rule: for $r,s\ge0$ with $r+s>0$,
\[
 \delta(x^{(r)})x^{(s)}+x^{(r)}\delta(x^{(s)})
 =\left(\binom{r+s-1}{r-1}+\binom{r+s-1}{r}\right)x^{(r+s-1)}
 =\delta(x^{(r)}x^{(s)}).
\]
The derivation preserves every $\OO(1;n)$, and we denote its restriction
by the same symbol.  In particular, $\delta=0$ on $\OO(1;0)$.
Its iterates satisfy
\begin{equation}
 \delta^j(x^{(r)})=
 \begin{cases}
  x^{(r-j)},&0\le j\le r,\\
  0,&j>r.
 \end{cases}
 \label{eq:lowering-iterates}
\end{equation}
Thus $\ker\delta=\kk1$ in $\OO(1)$, and on $\OO(1;n)$
the operator $\delta$ has a single nilpotent Jordan block of size $p^n$.
For a derivation $\partial$ of a commutative associative algebra
$\mathcal A$, we call $\mathcal A$ \emph{$\partial$-simple} if it has
no nonzero proper $\partial$-stable ideal.
If a nonzero $\delta$-stable ideal of $\OO(1;n)$ contains
$f=\sum_{r=0}^s c_rx^{(r)}$ with $c_s\ne0$, it contains
$\delta^s(f)=c_s1$.  Hence $\OO(1;n)$ is $\delta$-simple.

For $n>0$, let $\partial_i=\partial/\partial x_i$ be the coordinate
derivations, so that $\partial_i(x_j)=\delta_{ij}$.
The relations $x_j^p=0$ are preserved because $p=0$ in $\kk$.
Every derivation is uniquely determined by its values on the $x_i$,
and therefore
\[
 \Der(\OO(1;n))=\bigoplus_{i=1}^n\OO(1;n)\partial_i.
\]
Formula~\eqref{eq:divided-power-digits} and $(p-1)!=-1$ in $\kk$ give
\begin{equation}
 \delta(x_i)=x^{(p^{i-1}-1)}
 =(-1)^{i-1}\prod_{j<i}x_j^{p-1},
 \qquad
 \delta=\sum_{i=1}^n(-1)^{i-1}
 \left(\prod_{j<i}x_j^{p-1}\right)\partial_i,
 \label{eq:lowering-polynomial-coordinates}
\end{equation}
where an empty product is $1$.
Set $\Delta_0=0$ and, for $1\le r\le n$, define
\begin{equation}
 \Delta_r
 =\partial_1+x_1^{p-1}\partial_2+\cdots+
   x_1^{p-1}\cdots x_{r-1}^{p-1}\partial_r.
 \label{eq:Albert-r}
\end{equation}
This derivation acts trivially on $x_{r+1},\ldots,x_n$; for $r=n$
it is the Albert derivation~\cite{Albert1953,Block1969}.
Define $\tau_r\in\Aut(\OO(1;r))$ by $\tau_r(x_i)=(-1)^{i-1}x_i$.
Since $p$ is odd, $\tau_r(x_j^{p-1})=x_j^{p-1}$, so on $\OO(1;r)$
\begin{equation}
 \tau_r\delta\tau_r^{-1}=\Delta_r.
 \label{eq:albert-lowering-conjugacy}
\end{equation}
Consequently, $\Delta_r$ acts on $\OO(1;r)$ as a single nilpotent
Jordan block of size $p^r$, and $\OO(1;r)$ is $\Delta_r$-simple.

The \emph{divided-power maps} $f\mapsto f^{(r)}$ on
$\bigoplus_{m\ge1}\kk x^{(m)}$ are determined by
\begin{equation}
 \begin{gathered}
  \bigl(x^{(m)}\bigr)^{(r)}
  =\frac{(rm)!}{r!(m!)^r}x^{(rm)}\quad(m\ge1),\\
  (cf)^{(r)}=c^rf^{(r)},\qquad
  (f+g)^{(r)}=\sum_{i=0}^r f^{(i)}g^{(r-i)}
  \quad(r\ge0).
 \end{gathered}
 \label{eq:divided-powers-of-elements}
\end{equation}
Here $f,g$ have zero constant coefficient and $c\in\kk$.
The factorial quotient is an integer: it counts partitions of an
$rm$-element set into $r$ unordered blocks of size $m$.
It is evaluated in $\mathbb Z$ before reduction to $\kk$.
The sum and scalar rules extend the monomial formula to every finite
linear combination, with $f^{(0)}=1$ and $f^{(1)}=f$.
For $0\le r<p$, this gives $f^{(r)}=f^r/r!$.

\begin{lemma}
\label{lem:divided-power-identities}
For $f,g\in\OO(1)$ with zero constant coefficient and integers
$r,s\ge0$,
\[
 f^{(r)}f^{(s)}=\binom{r+s}{r}f^{(r+s)},
 \qquad
 (fg)^{(s)}=f^sg^{(s)},
\]
and
\begin{equation}
 \delta\bigl(f^{(r)}\bigr)=\delta(f)f^{(r-1)}\qquad(r\ge1).
 \label{eq:full-divided-power-chain-rule}
\end{equation}
\end{lemma}

\begin{proof}
Regard the coefficients of $f$ and $g$ as indeterminates over
$\mathbb Z$.  Over $\mathbb Q$, use $x^{(m)}=x^m/m!$ and
$f^{(r)}=f^r/r!$.  All coefficients in the divided-power basis are
integral polynomials in these indeterminates, so the identities remain
valid over $\kk$.
\end{proof}

\begin{remark}
By~\eqref{eq:full-divided-power-chain-rule}, $\delta$ is a
\emph{special derivation}~\cite[\S3.2]{PremetStrade2006}.
We use $\Der(\OO(1;n))$ for all derivations of the associative algebra.
The divided powers $f^{(r)}$ are computed in $\OO(1)$,
even when $f\in\OO(1;n)$.  They need not lie in $\OO(1;n)$:
for example, $x^{(p)}\ne0$ in $\OO(1)$ whereas $x^p=0$,
and $x^{(p)}\notin\OO(1;1)$.
\end{remark}

\subsection{Simple Novikov algebras and their isomorphisms}
\label{subsec:twisted-GD-construction}

\begin{definition}
\begin{enumerate}[label=\textup{(\roman*)},leftmargin=2em]
\item
A \emph{Novikov algebra} is a vector space $\mathcal{N}$ with a bilinear
product $\circ$ satisfying, for all $x,y,z\in\mathcal{N}$,
\begin{align}
 (x\circ y)\circ z-x\circ(y\circ z)
  &=(y\circ x)\circ z-y\circ(x\circ z),
 \label{eq:left-symmetry}\\
 (x\circ y)\circ z&=(x\circ z)\circ y.
 \label{eq:right-commutativity}
\end{align}
\item
A Novikov algebra $\mathcal{N}$ is \emph{simple} if
$\mathcal{N}\circ\mathcal{N}\ne0$ and $\mathcal{N}$ has no nonzero proper
two-sided ideal.
\item
For a Novikov algebra $\mathcal{N}$, the left and right multiplication
operators are $\Lop_x(y)=x\circ y$ and $\Rop_x(y)=y\circ x$.
\end{enumerate}
\end{definition}

By~\eqref{eq:right-commutativity}, the right multiplications commute. 
Denote by the associative algebra
\[ \mathcal R(\mathcal N):=
 \operatorname{span}_{\kk}
 \{\Rop_{a_1}\cdots\Rop_{a_r}:r\ge1,\ a_i\in\mathcal{N}\}.
\]

\begin{definition}
\label{def:twisted-GD-construction}
Let $\mathcal{A}$ be a commutative associative algebra,
$\partial\in\Der(\mathcal{A})$, and $a\in\mathcal{A}$.
The \emph{twisted Gelfand--Dorfman construction}~\cite{Xu1997,Xu2001}
assigns to $(\mathcal{A},\partial,a)$ the Novikov algebra
$\mathcal{A}(\partial,a)$ on the vector space $\mathcal{A}$ with product
\begin{equation}
 x\circ y=x\bigl(\partial(y)+ay\bigr)
 \qquad(x,y\in\mathcal{A}).
 \label{eq:GD-product}
\end{equation}
For $a=0$, this is the Gelfand--Dorfman construction~\cite{GD1979}.
\end{definition}

\begin{theorem}[Zhelyabin and Zakharov {\cite[Theorem~3, Corollary~1]{ZZ2024}}]
\label{thm:structural-reduction}
Every finite-dimensional simple Novikov algebra $\mathcal{N}$ is
isomorphic to $\OO(1;n)(\partial,a)$ for some $n\ge0$,
$\partial\in\Der(\OO(1;n))$, and $a\in\OO(1;n)$, where
$\OO(1;n)$ is $\partial$-simple.
The associative algebra 
$\mathcal R(\mathcal{N})$ is also isomorphic to $\OO(1;n)$.
\end{theorem}

\begin{lemma}
\label{lem:normalized-simple}
Let $n\ge0$ and $b\in \mathcal{A}=\OO(1;n)$, with $b=1$ if $n=0$.
The algebra $\mathcal{N}=\mathcal{A}(\delta,b)$ is simple, and $\mathcal R(\mathcal{N})$ consists
of all operators $m_c:x\mapsto cx$, $c\in\mathcal{A}$.
\end{lemma}

\begin{proof}
For $n=0$, the algebra is $\kk$ and $R_1=\Id$.  Suppose $n>0$,
and put $q=p^n$ and $D=\delta+m_b$.
Since $R_y=m_{D(y)}$ and $m_cm_d=m_{cd}$, for the last assertion, it suffices to show that
$D(\mathcal{A})={\rm span}\{  D(x)\mid x\in \mathcal A\}$ generates $\mathcal{A}$ as an associative algebra.
For $1\le r<q$, the element
\[
 D(x^{(r)})=x^{(r-1)}+bx^{(r)}
\]
has lowest-degree term $x^{(r-1)}$.
To obtain the highest-degree element, we use $q\ge3$ and
\[
 D(x^{(2)})D(x^{(q-1)})
 =x\,x^{(q-2)}=-x^{(q-1)},
\]
since all terms containing $b$ have degree at least $q$ and vanish,
and $\binom{q-1}{1}=-1$ in $\kk$.
These $q$ elements have distinct lowest degrees, so they form a basis
of $\mathcal{A}$.

Let $I$ be a Novikov ideal of $\mathcal{N}$.
It is stable under every $m_c$, so it is an associative ideal
of $\mathcal{A}$.  It is stable under $D=L_1$ and therefore under
$\delta=D-m_b$.  Since $\mathcal{A}$ is $\delta$-simple, $I=0$ or $I=\mathcal{A}$.
Finally, $1\circ x=1+bx\ne0$, so $\mathcal{N}$ is simple.
\end{proof}

\Needspace{13\baselineskip}
\begin{theorem}[Pozhidaev, Zakharov and Zhelyabin {\cite[Theorem~3.3, Corollary~3.1]{PZZ2025}}]
\label{thm:change-data}
Let $\mathcal{A}(\partial,a)$ and $\mathcal{A}'(\partial',a')$ be finite-dimensional simple
Novikov algebras, where $\mathcal{A}$ and $\mathcal{A}'$ are unital commutative
associative algebras.
Every Novikov algebra isomorphism
$\phi:\mathcal{A}(\partial,a)\to\mathcal{A}'(\partial',a')$
has a unique expression
\begin{equation}
 \phi(x)=h^{-1}\sigma(x)\qquad(x\in\mathcal{A}),
 \label{eq:all-isomorphisms-form}
\end{equation}
where $\sigma:\mathcal{A}\to \mathcal{A}'$ is an associative algebra isomorphism and
$h\in(\mathcal{A}')^\times$.  With $\partial^\sigma=\sigma \partial\sigma^{-1}$, the data satisfy
\begin{equation}
 \partial'=h \partial^\sigma,
 \qquad a'=\partial^\sigma(h)+h\sigma(a).
 \label{eq:all-isomorphisms-data}
\end{equation}
Conversely, any $\sigma$ and $h$ satisfying~\eqref{eq:all-isomorphisms-data} give an
isomorphism by~\eqref{eq:all-isomorphisms-form}.
\end{theorem}

\Needspace{8\baselineskip}
\section{Normalization of derivations}
\label{sec:derivation-normalization}

We reduce the algebras in Theorem~\ref{thm:structural-reduction} to
$\OO(1;n)(\delta,b)$ by coordinate changes and multiplication by units.
We use the polynomial coordinates $x_i$ and the derivations
$\partial_i,\Delta_r$ from Section~\ref{subsec:divided-power-algebras}.

For $\partial\in\Der(\OO(1;n))$, write
\[
 \partial=\partial_{\mathrm{ss}}+\partial_{\mathrm{nil}},
 \qquad [\partial_{\mathrm{ss}},\partial_{\mathrm{nil}}]=0
\]
for its Jordan decomposition as a linear operator.
Both components are derivations.  Indeed, the Leibniz rule gives
\[
 \bigl(\partial-(\lambda+\mu)\Id\bigr)^m(fg)
 =\sum_{i=0}^m\binom mi
   (\partial-\lambda\Id)^i(f)
   (\partial-\mu\Id)^{m-i}(g).
\]
If $(\partial-\lambda\Id)^s(f)=0$ and
$(\partial-\mu\Id)^t(g)=0$, the right-hand side vanishes for
$m\ge s+t-1$.  Thus the product of generalized eigenvectors of
eigenvalues $\lambda,\mu$ lies in the generalized eigenspace of
$\lambda+\mu$, and
\[
 \partial_{\mathrm{ss}}(fg)
 =(\lambda+\mu)fg
 =\partial_{\mathrm{ss}}(f)g+f\partial_{\mathrm{ss}}(g).
\]
The assertion for $\partial_{\mathrm{nil}}=\partial-\partial_{\mathrm{ss}}$
follows.

For $(p,n)\ne(3,1)$, the next normal form follows
from~\cite[Theorem~2(ii)--(iii)]{Premet2015}.
We give a direct proof for all $p>2$.

\begin{lemma}
\label{lem:d-simple-derivation-form}
For $n>0$, let $\partial\in\Der(\OO(1;n))$ be such that $\OO(1;n)$ is $\partial$-simple.
Then there exist an integer $r$ with $0\le r\le n$, an automorphism
$\sigma\in\Aut(\OO(1;n))$, and elements
$\nu_{r+1},\ldots,\nu_n\in\kk$ linearly independent over
$\mathbb F_p$ such that
\begin{equation}
 \sigma \partial\sigma^{-1}
 =\Delta_r+\sum_{i=r+1}^n
   \nu_i(1+x_i)\partial_i.
 \label{eq:d-simple-derivation-form}
\end{equation}
The integer $r$ is determined by
$\dim\ker \partial_{\mathrm{ss}}=p^r$.
\end{lemma}

\begin{proof}
Write $\mathcal A=\OO(1;n)$.
Differential simplicity first determines the kernel of $\partial$:
if $f\in\ker\partial$, then
$\mathcal A(f-f(0))\subseteq\mathfrak m$ is a proper $\partial$-stable ideal.
Hence $f=f(0)$ and $\ker\partial=\kk1$.
Let $\Lambda$ be the spectrum of $\partial$, and put
\[
 \mathcal A_\lambda=\ker(\partial-\lambda\Id)^{p^n},
 \qquad
 \mathcal A=\bigoplus_{\lambda\in\Lambda}\mathcal A_\lambda.
\]
In particular, $\mathcal A_0=\ker\partial_{\mathrm{ss}}$ is a unital
$\partial$-stable subalgebra.

To compare the generalized eigenspaces with $\mathcal A_0$, we use
the fact that every nonzero eigenvector is a unit.  Indeed, if
$0\ne w\in\ker(\partial-\lambda\Id)$, then
\[
 \partial(aw)=\bigl(\partial(a)+\lambda a\bigr)w
 \qquad(a\in\mathcal A).
\]
Thus $\mathcal Aw$ is a nonzero $\partial$-stable ideal, so
$\mathcal Aw=\mathcal A$ and $w$ is a unit.
We may therefore choose eigenvectors $z_\lambda$ with $z_\lambda(0)=1$.
Then $z_\lambda^p=1$ and
\[
 (\partial-\lambda\Id)^j(z_\lambda f)
 =z_\lambda\partial^j(f)
 \qquad(f\in\mathcal A,\ j\ge0).
\]
Since $z_\lambda$ is a unit, this gives
$\mathcal A_\lambda=z_\lambda\mathcal A_0$.
Moreover,
\[
 \partial(z_\lambda z_\mu)=(\lambda+\mu)z_\lambda z_\mu,
 \qquad
 \partial(z_\lambda^{-1})=-\lambda z_\lambda^{-1}.
\]
The products and inverses are nonzero, so $\Lambda$ is a finite
$\mathbb F_p$-subspace of $\kk$.  Consequently, for some $0\le r\le n$,
\[
 p^n=|\Lambda|\dim\mathcal A_0,
 \qquad \dim\mathcal A_0=p^r,
 \qquad \dim_{\mathbb F_p}\Lambda=n-r.
\]
Choose an $\mathbb F_p$-basis $\nu_{r+1},\ldots,\nu_n$ of $\Lambda$
and eigenvectors $z_i=z_{\nu_i}$.  For $0\le a_i<p$,
\[
 \partial\left(\prod_{i=r+1}^n z_i^{a_i}\right)
 =\left(\sum_{i=r+1}^n a_i\nu_i\right)
   \prod_{i=r+1}^n z_i^{a_i}.
\]
These eigenvalues exhaust $\Lambda$ without repetition, so the
generalized eigenspace decomposition takes the form
\[
 \mathcal A=\bigoplus_{0\le a_{r+1},\ldots,a_n<p}
 \mathcal A_0\prod_{i=r+1}^n z_i^{a_i}.
\]

This decomposition reduces the construction of coordinates to
the subalgebra $\mathcal A_0$.
The restriction $\partial_0=\partial|_{\mathcal A_0}$ is nilpotent
with one-dimensional kernel, so it has a single Jordan block and
\[
 \partial_0^{p^r}=0,
 \qquad \ker\partial_0=\kk1,
 \qquad \operatorname{im}\partial_0=\ker\partial_0^{p^r-1}.
\]
Starting with $\psi_0:\kk\hookrightarrow\mathcal A_0$, we construct
injective unital algebra homomorphisms
\[
 \psi_j:\OO(1;j)\longrightarrow\mathcal A_0,
 \qquad x_i\longmapsto u_i\in\mathcal A_0\cap\mathfrak m,
 \qquad \partial_0\psi_j=\psi_j\Delta_j.
\]
Suppose $\psi_j$ has been constructed and $j<r$.
For the next coordinate, we need a preimage under $\partial_0$ of
$t_j=\prod_{i=1}^j u_i^{p-1}$, with $t_0=1$.
Since $\Delta_j^{p^j}=0$ on $\OO(1;j)$,
\[
 \partial_0^{p^j}(t_j)
 =\psi_j\left(\Delta_j^{p^j}
   \left(\prod_{i=1}^j x_i^{p-1}\right)\right)=0.
\]
As $p^j\le p^r-1$, we obtain
\[
 t_j\in\ker\partial_0^{p^j}
 \subseteq\ker\partial_0^{p^r-1}
 =\operatorname{im}\partial_0.
\]
Choose $v\in\mathcal A_0$ with $\partial_0(v)=t_j$ and subtract its
constant term to obtain
\[
 u_{j+1}=v-v(0)\in\mathcal A_0\cap\mathfrak m,
 \qquad u_{j+1}^p=0,
 \qquad \partial_0(u_{j+1})=t_j.
\]
Thus $x_{j+1}\mapsto u_{j+1}$ extends $\psi_j$ to a unital algebra
homomorphism $\psi_{j+1}$ satisfying
$\partial_0\psi_{j+1}=\psi_{j+1}\Delta_{j+1}$.
Its kernel is a proper $\Delta_{j+1}$-stable ideal, so it is zero by
Section~\ref{subsec:divided-power-algebras}.
This completes the induction.  Since
$\dim\OO(1;r)=\dim\mathcal A_0=p^r$, we have
$\mathcal A_0=\kk[u_1,\ldots,u_r]$.

The remaining coordinates come from the chosen eigenvectors:
for $i>r$, put $u_i=z_i-1$.  Then
\[
 u_i^p=z_i^p-1=0,
 \qquad \partial(u_i)=\nu_i(1+u_i).
\]
The preceding decomposition gives $\mathcal A=\kk[u_1,\ldots,u_n]$.
Hence $\phi(x_i)=u_i$ defines a surjective algebra endomorphism of
$\mathcal A$, which is an automorphism by finite-dimensionality.
On the generators,
\[
 \phi^{-1}\partial\phi(x_i)=
 \begin{cases}
  \displaystyle\prod_{\ell=1}^{i-1}x_\ell^{p-1},&1\le i\le r,\\[4pt]
  \nu_i(1+x_i),&r<i\le n.
 \end{cases}
\]
Thus $\sigma=\phi^{-1}$ gives~\eqref{eq:d-simple-derivation-form}.
\end{proof}

\Needspace{10\baselineskip}
\begin{corollary}
\label{cor:d-simple-derivation-orbit}
Let $n>0$ and $\partial\in\Der(\OO(1;n))$.  Then $\OO(1;n)$ is $\partial$-simple if and only if
there exist $\sigma\in\Aut(\OO(1;n))$ and $h\in \OO(1;n)^\times$ such that
\begin{equation}
 h\,\sigma \partial\sigma^{-1}=\Delta_n.
 \label{eq:single-derivation-orbit}
\end{equation}
\end{corollary}

\begin{proof}
Write $\mathcal A=\OO(1;n)$ and suppose that $\mathcal A$ is
$\partial$-simple.  By Lemma~\ref{lem:d-simple-derivation-form},
after conjugation we may assume that
\begin{equation}
 \partial=\Delta_r+\sum_{i=r+1}^n\nu_i(1+x_i)\partial_i,
 \label{eq:absorption-start}
\end{equation}
where $\nu_{r+1},\ldots,\nu_n$ are linearly independent over $\mathbb F_p$.
If $r=n$, then $\partial=\Delta_n$.
Suppose $r<n$.  We will use a unit multiple of $\partial$ to increase
$r$ by one in the normal form.  Put
\[
 u=1+x_{r+1},\qquad \nu=\nu_{r+1}\ne0,\qquad E=u\partial.
\]
For every ideal $I$ of $\mathcal A$, the fact that $u$ is a unit gives
\[
 E(I)\subseteq I\quad\Longleftrightarrow\quad\partial(I)\subseteq I.
\]
Thus $\mathcal A$ is $E$-simple.
Since $\partial(u)=\nu u$ and $u^p=1$, we have
\[
 \partial m_u=m_u(\partial+\nu\Id),
 \qquad
 E^s=m_{u^s}\prod_{j=0}^{s-1}(\partial+j\nu\Id)
 \quad(s\ge1),
\]
where the second identity follows by induction.
Taking $s=p$ gives
\begin{equation}
 E^p=\prod_{c\in\mathbb F_p}(\partial+c\nu\Id)
     =\partial^p-\nu^{p-1}\partial,
 \label{eq:eigenunit-p-power}
\end{equation}
using $\prod_{c\in\mathbb F_p}(X+c\nu)=X^p-\nu^{p-1}X$.

To determine how this multiplication changes $r$, we compute the
generalized zero eigenspace of $E$.
By Lemma~\ref{lem:d-simple-derivation-form} and its proof,
\[
 \Lambda=\operatorname{span}_{\mathbb F_p}\{\nu_{r+1},\ldots,\nu_n\},
 \qquad \mathcal A=\bigoplus_{\lambda\in\Lambda}\mathcal A_\lambda,
\]
and on each generalized eigenspace,
\[
 \dim\mathcal A_\lambda=p^r,
 \qquad
 \left.(\partial-\lambda\Id)^{p^r}\right|_{\mathcal A_\lambda}=0.
\]
Since $r\le n-1$, equation~\eqref{eq:eigenunit-p-power} yields
\begin{align*}
 \left.E^{p^n}\right|_{\mathcal A_\lambda}
 &=\left.(\partial^p-\nu^{p-1}\partial)^{p^{n-1}}
   \right|_{\mathcal A_\lambda}\\
 &=\bigl(\lambda^{p^n}-\nu^{(p-1)p^{n-1}}
                 \lambda^{p^{n-1}}\bigr)\Id\\
 &=\bigl(\lambda^p-\nu^{p-1}\lambda\bigr)^{p^{n-1}}\Id.
\end{align*}
The scalar vanishes exactly when $\lambda\in\mathbb F_p\nu$.
Since $\dim\mathcal A=p^n$, we therefore obtain
\begin{equation}
 \begin{gathered}
  \ker E_{\mathrm{ss}}=\ker E^{p^n}
       =\bigoplus_{c\in\mathbb F_p}\mathcal A_{c\nu},\\
  \dim\ker E_{\mathrm{ss}}=p\cdot p^r=p^{r+1}.
 \end{gathered}
 \label{eq:toral-rank-drop}
\end{equation}
Applying Lemma~\ref{lem:d-simple-derivation-form} to $E$ replaces $r$
by $r+1$ in~\eqref{eq:absorption-start}.
Repeating this step $n-r$ times gives $\Delta_n$.
The successive changes combine into a single change of the required
form: for automorphisms $\sigma,\rho$ and units $h,g$, we have
\[
 g(h\partial^\sigma)^\rho
 =g\,\rho(h)\,\partial^{\rho\sigma}
\]
and $g\,\rho(h)$ is again a unit.

Conversely, if $h\partial^\sigma=\Delta_n$, then for every ideal $I$
of $\mathcal A$,
\[
 \partial(I)\subseteq I
 \quad\Longleftrightarrow\quad
 \Delta_n\bigl(\sigma(I)\bigr)\subseteq\sigma(I).
\]
The $\Delta_n$-simplicity of $\mathcal A$ gives $I=0$ or $I=\mathcal A$.
\end{proof}

\begin{proposition}
\label{prop:normalized-differential-data}
Every finite-dimensional simple Novikov algebra over $\kk$ is
isomorphic to
\begin{equation}
 \OO(1;n)(\delta,b),\qquad
 x\circ y=x\bigl(\delta(y)+by\bigr),
 \label{eq:fixed-Albert-normal-form}
\end{equation}
for some $n\ge0$ and $b\in\OO(1;n)$, with $b=1$ if $n=0$.
\end{proposition}

\begin{proof}
For $\mathcal N\cong\kk$, take $n=0$ and $b=1$.
By Theorem~\ref{thm:structural-reduction}, it remains to consider
$\OO(1;n)(\partial,a)$ with $n>0$ and $\OO(1;n)$ $\partial$-simple.
Corollary~\ref{cor:d-simple-derivation-orbit} gives
$\sigma_0\in\Aut(\OO(1;n))$ and $h_0\in\OO(1;n)^\times$ with
$h_0\partial^{\sigma_0}=\Delta_n$.
Using $\tau_n\delta\tau_n^{-1}=\Delta_n$
from~\eqref{eq:albert-lowering-conjugacy}, put
\[
 \sigma=\tau_n^{-1}\sigma_0,
 \qquad h=\tau_n^{-1}(h_0).
\]
Then
\[
 h\partial^\sigma
 =\tau_n^{-1}(h_0\partial^{\sigma_0})\tau_n
 =\tau_n^{-1}\Delta_n\tau_n
 =\delta.
\]
Set
\begin{equation}
 b=\partial^\sigma(h)+h\sigma(a)
  =h^{-1}\delta(h)+h\sigma(a).
 \label{eq:normalized-b}
\end{equation}
By Lemma~\ref{lem:normalized-simple}, $\OO(1;n)(\delta,b)$ is simple.
Theorem~\ref{thm:change-data} now gives
\[
 h^{-1}\sigma:\OO(1;n)(\partial,a)
 \xrightarrow{\ \sim\ }\OO(1;n)(\delta,b).\qedhere
\]
\end{proof}

\Needspace{8\baselineskip}

\section{Classification of simple algebras}
\label{sec:residual-orbit}

By Proposition~\ref{prop:normalized-differential-data}, the classification
of finite-dimensional simple Novikov algebras reduces to that of
$\OO(1;n)(\delta,b)$.  We use admissible automorphisms to reduce $b$ to
$b_{\alpha,\beta}$ and prove that all isomorphisms between the resulting
normal forms arise from scalings.  This proves Theorem~A and also yields
their automorphism groups and a classification of Novikov--Poisson
algebras with simple Novikov part.

\subsection{Normal forms}

Let $n>0$.  By Lemma~\ref{lem:normalized-simple} and
Theorem~\ref{thm:change-data}, the isomorphisms among the algebras
$\OO(1;n)(\delta,b)$, $b\in\OO(1;n)$, are the maps $h^{-1}\sigma$
with $(\sigma,h)\in G_n$, where
\begin{equation}
 G_n=
 \bigl\{(\sigma,h)\in\Aut(\OO(1;n))\times\OO(1;n)^\times:
 h\delta^\sigma=\delta\bigr\},
 \qquad \delta^\sigma=\sigma\delta\sigma^{-1}.
 \label{eq:residual-group}
\end{equation}
This is a group under the multiplication
\[
 (\tau,g)(\sigma,h)=(\tau\sigma,g\tau(h)).
\]
Indeed,
\[
 (g^{-1}\tau)(h^{-1}\sigma)
 =(g\tau(h))^{-1}\tau\sigma.
\]
The corresponding action on $\OO(1;n)$ is
\begin{equation}
 (\sigma,h)\mathbin{\ast}b
 =\delta^\sigma(h)+h\sigma(b).
 \label{eq:residual-b-action}
\end{equation}

An automorphism $\sigma$ of $\OO(1;n)$ is \emph{admissible} if
$\sigma(x^{(r)})=\sigma(x)^{(r)}$ in $\OO(1)$ for all $0\le r<p^n$.
Wilson described these automorphisms for $p\ge5$ in
\cite[Lemmas~9--10 and Corollaries~1--2 to Lemma~9]{Wilson1971}.
We prove the following criterion for $p>2$.

\begin{proposition}
\label{prop:residual-admissible-group}
The map $(\sigma,h)\mapsto\sigma$ is an isomorphism from $G_n$ onto
the group of admissible automorphisms of $\OO(1;n)$.  The map
$\sigma\mapsto\sigma(x)$ is a bijection from this group onto the set of
elements
\begin{equation}
 f=\sum_{r=1}^{p^n-1}a_rx^{(r)},
 \qquad
 a_1\ne0,\qquad
 a_{p^j}=0\quad(1\le j<n).
 \label{eq:admissible-substitution}
\end{equation}
The corresponding second component is
\begin{equation}
 h=\delta(f),
 \label{eq:residual-Jacobian}
\end{equation}
and the action \eqref{eq:residual-b-action} is
\begin{equation}
 \sigma\mathbin{\ast}b
 =h^{-1}\delta(h)+h\sigma(b).
 \label{eq:connection-form-residual-action}
\end{equation}
\end{proposition}

\begin{proof}
Every automorphism $\sigma$ of $\OO(1;n)$ preserves its maximal ideal, so
$f=\sigma(x)$ has zero constant coefficient.
The condition $h\delta^\sigma=\delta$ is equivalent to
\begin{equation}
 \delta\sigma=h\sigma\delta.
 \label{eq:chain-rule-residual-group}
\end{equation}
Applying it to $x$ gives $h=\delta(f)$.  We prove
$\sigma(x^{(r)})=f^{(r)}$ by induction on $r$.  The case $r=0$ is immediate.
For $1\le r<p^n$, the induction hypothesis,
\eqref{eq:chain-rule-residual-group}, and
Lemma~\ref{lem:divided-power-identities} give
\[
 \delta\bigl(\sigma(x^{(r)})-f^{(r)}\bigr)
 =h\sigma(x^{(r-1)})-\delta(f)f^{(r-1)}=0.
\]
Since $\ker\delta=\kk1$ in $\OO(1)$ and both terms have zero
constant coefficient, their difference is zero.  Thus $\sigma$ is admissible.
Conversely, let $\sigma$ be an admissible automorphism and write
$f=\sigma(x)=\sum_{m=1}^{p^n-1}a_mx^{(m)}$.  The divided-power expansion
shows that $f^{(r)}$ has no terms of degree less than $r$, and its
degree-$r$ coefficient is $a_1^r$.  Thus the matrix of $\sigma$ is triangular
with diagonal $1,a_1,\ldots,a_1^{p^n-1}$.  Invertibility gives $a_1\ne0$,
so $h=\delta(f)$ is a unit.  The chain rule gives
\eqref{eq:chain-rule-residual-group}, and hence $(\sigma,h)\in G_n$.
This proves the asserted identification of groups and
\eqref{eq:residual-Jacobian}.  Substituting $\delta^\sigma=h^{-1}\delta$
in~\eqref{eq:residual-b-action} gives
\eqref{eq:connection-form-residual-action}.

The restrictions on the coefficients of $f$ arise from requiring
$f^{(r)}\in\OO(1;n)$ for $0\le r<p^n$.  We obtain the necessary
restrictions by examining the coefficient of $x^{(p^n)}$.
For any two elements $u,v$ of
$\OO(1)$ with zero constant coefficient,
$[x^{(p^n)}](uv)=0$, since $\binom{p^n}{i}=0$ in $\kk$ for $0<i<p^n$.
Fix $1\le j<n$ and put $r=p^{n-j}$ and $m=p^j$, so that $rm=p^n$.
Every mixed term in the divided-power expansion of $f^{(r)}$ contains at
least two factors with zero constant coefficient and hence contributes
nothing to $x^{(p^n)}$.  Among the pure terms, only
$(a_mx^{(m)})^{(r)}$ can have degree $p^n$.  Admissibility therefore gives
\begin{equation}
 0=[x^{(p^n)}]f^{(r)}=\overline C_j a_{p^j}^{p^{n-j}},
 \qquad
 C_j=\frac{(p^n)!}{r!(m!)^r}\in\mathbb Z,
 \label{eq:forbidden-coefficient-expansion}
\end{equation}
where $\overline C_j$ is the image of $C_j$ in $\kk$.
Counting multiples of $p,p^2,\ldots$ shows that the exponents of $p$ in
the numerator and denominator of $C_j$ are, respectively,
\[
 \frac{p^n-1}{p-1}
 \quad\text{and}\quad
 \frac{r-1}{p-1}+r\frac{m-1}{p-1}
 =\frac{rm-1}{p-1}=\frac{p^n-1}{p-1}.
\]
Hence $p\nmid C_j$, so $\overline C_j\ne0$, and
\eqref{eq:forbidden-coefficient-expansion} implies $a_{p^j}=0$.

Conversely, suppose that $f$ satisfies
\eqref{eq:admissible-substitution}.  To construct the corresponding
automorphism, we first verify that
$f^{(r)}\in\OO(1;n)$ for $0\le r<p^n$.
If $1<m<p^n$ is not a power of $p$, choose a nonzero digit $m_j$ in its
base-$p$ expansion.  Then $m-p^j>0$, and Lucas's congruence gives
\[
 x^{(m)}=m_j^{-1}x^{(p^j)}x^{(m-p^j)}.
\]
Since $(x^{(p^j)})^p=p!\,(x^{(p^j)})^{(p)}=0$, the identity
$(uv)^{(s)}=u^sv^{(s)}$ gives
\[
 \bigl(x^{(m)}\bigr)^{(s)}=0\quad(s\ge p),
 \qquad
 \bigl(x^{(m)}\bigr)^{(s)}=\frac{(x^{(m)})^s}{s!}\in\OO(1;n)
 \quad(0\le s<p).
\]
In the divided-power expansion of $f^{(r)}$, every factor arising from a
nonzero term $a_mx^{(m)}$ with $m>1$ therefore lies in $\OO(1;n)$.
The factors arising from $a_1x$ are $a_1^s x^{(s)}$ with $s\le r<p^n$,
so they also lie in $\OO(1;n)$.  Since $\OO(1;n)$ is a subalgebra, the claim follows.

We can therefore define $\sigma(x^{(r)})=f^{(r)}$ for $0\le r<p^n$.
The divided-power product identity, together with
$\binom{r+s}{r}=0$ when $r,s<p^n\le r+s$, shows that
$\sigma$ is an algebra endomorphism.  The triangularity established above
and $a_1\ne0$ imply that $\sigma$ is invertible.  It is admissible by
construction and uniquely determined by $f$.
\end{proof}

We write $\sigma$ for $(\sigma,\delta(\sigma(x)))\in G_n$ and use
the filtration $F^s$ from~\eqref{eq:divided-power-filtration}.

\begin{lemma}
\label{lem:first-changed-coefficient}
Let $\sigma\ne\Id$ be an admissible automorphism, with coefficients
$a_s$ as in~\eqref{eq:admissible-substitution} and $a_1=1$.  Put
$r=\min\{s:2\le s<p^n,\ a_s\ne0\}$.  Then, for every $b\in\OO(1;n)$,
\begin{equation}
 \sigma\mathbin{\ast}b
 \equiv b+a_rx^{(r-2)}\pmod{F^{r-1}}.
 \label{eq:first-changed-coefficient}
\end{equation}
\end{lemma}

\begin{proof}
Put $f=\sigma(x)$ and $h=\delta(f)$.
We estimate the two terms in~\eqref{eq:connection-form-residual-action}
separately.  By the choice of $r$,
\[
 h\equiv1+a_rx^{(r-1)}\pmod{F^r},
 \qquad
 \delta(h)\equiv a_rx^{(r-2)}\pmod{F^{r-1}}.
\]
Since $h^{-1}-1\in F^{r-1}$ and $r\ge2$,
\[
 (h^{-1}-1)\delta(h)\in F^{2r-3}\subseteq F^{r-1},
\]
so
\begin{equation}
 h^{-1}\delta(h)
 \equiv a_rx^{(r-2)}\pmod{F^{r-1}}.
 \label{eq:log-Jacobian-leading-term}
\end{equation}
For the term $h\sigma(b)$, write $u=f-x\in F^r$.
The divided-power expansion in $\OO(1)$ gives, for $1\le s<p^n$,
\[
 f^{(s)}-x^{(s)}
 =\sum_{i=1}^s x^{(s-i)}u^{(i)}.
\]
Every term in the $i$th summand has degree at least
$(s-i)+ir\ge s+r-1$; hence
$\sigma(x^{(s)})-x^{(s)}\in F^{s+r-1}$.
Since $\sigma(1)=1$, it follows that $\sigma(b)-b\in F^r$ and
$h\sigma(b)-b\in F^{r-1}$.
Combining these inclusions with~\eqref{eq:log-Jacobian-leading-term}
proves~\eqref{eq:first-changed-coefficient}.
\end{proof}

Define the $(n+1)$-dimensional subspace
\begin{equation}
 S_n=
 \bigoplus_{j=1}^n\kk x^{(p^j-2)}
 \mathbin{\oplus}\kk x^{(p^n-1)},
 \label{eq:residual-slice}
\end{equation}
and let $G_n^1\subseteq G_n$ consist of the elements with $a_1=1$
in~\eqref{eq:admissible-substitution}.  This is a subgroup because
the coefficient $a_1$ is multiplicative under composition.

\Needspace{17\baselineskip}
\begin{proposition}
\label{prop:residual-orbit}
For every $b\in\OO(1;n)$, there is a unique $g\in G_n^1$ such that
$g\mathbin{\ast}b\in S_n$.
For $\alpha,\alpha'\in\kk^n$ and $\beta,\beta'\in\kk$, the elements
$b_{\alpha,\beta}$ and $b_{\alpha',\beta'}$
in~\eqref{eq:normal-form-parameters} lie in the same
$G_n$-orbit if and only if there is $c\in\kk^\times$ such that
\begin{equation}
 \alpha'_j=c^{p^j-1}\alpha_j\quad(1\le j\le n),
 \qquad
 \beta'=c^{p^n}\beta.
 \label{eq:weighted-parameter-equivalence}
\end{equation}
Every element of $G_n$ carrying one point of $S_n$ to another is
$(\rho_c,c)$ for some $c\in\kk^\times$, where
$\rho_c(x^{(r)})=c^r x^{(r)}$.
\end{proposition}

\begin{proof}
We eliminate the coefficients outside the degrees occurring in $S_n$,
starting with the lowest degree.
For each integer $m$ with $0\le m\le p^n-3$ and
\[
 m\notin\{p^j-2\mid 1\le j<n\},
\]
Proposition~\ref{prop:residual-admissible-group} gives an admissible
automorphism $\sigma_m$ with
\[
 \sigma_m(x)=x+c_mx^{(m+2)}.
\]
Lemma~\ref{lem:first-changed-coefficient} gives
\[
 [x^{(s)}](\sigma_m\ast b)=
 \begin{cases}
  [x^{(s)}]b,&0\le s<m,\\
  [x^{(m)}]b+c_m,&s=m.
 \end{cases}
\]
At each step, choose $c_m=-[x^{(m)}]b$ for the current $b$.
The displayed identity shows that this cancels the coefficient of
degree $m$ without changing the coefficients already eliminated.
Take $\sigma_m=\Id$ when $c_m=0$.
Their composition gives $g\in G_n^1$ with $g\mathbin{\ast}b\in S_n$.

For uniqueness, we show that an element of $G_n^1$ carrying one point
of $S_n$ to another must be the identity.
Suppose that $b,b'\in S_n$ and that $b'=\sigma\mathbin{\ast}b$, where
$\sigma\in G_n^1$.  If $\sigma$ is not the identity, put
$r=\min\{s\mid 2\le s<p^n,\ a_s\ne0\}$.
By~\eqref{eq:admissible-substitution}, $r\ne p^j$ for $1\le j<n$.
Since $r\le p^n-1$, we also have $r-2\le p^n-3$.
Thus $r-2$ is not one of the degrees occurring in $S_n$.  But
Lemma~\ref{lem:first-changed-coefficient} shows that the
coefficient of $x^{(r-2)}$ changes from $0$ to $a_r\ne0$, a contradiction.
Thus $\sigma$ is the identity.
If $g_1,g_2\in G_n^1$ both carry a given $b\in\OO(1;n)$ into $S_n$, then
$g_2g_1^{-1}$ sends $g_1\mathbin{\ast}b$ to $g_2\mathbin{\ast}b$.
It is therefore the identity, and $g_1=g_2$.

It remains to account for automorphisms whose linear coefficient
need not be $1$.
For $c\in\kk^\times$, let $\rho_c$ be the admissible automorphism
$\rho_c(x)=cx$.  Here $h=c$, so $\rho_c\mathbin{\ast}b=c\rho_c(b)$ and
\begin{equation}
 \rho_c\mathbin{\ast}x^{(s)}=c^{s+1}x^{(s)}.
 \label{eq:weighted-scaling-on-divided-powers}
\end{equation}
In particular, $\rho_c$ preserves $S_n$.
Every admissible automorphism $\sigma$ has a unique factorization
$\sigma=\rho_c\tau$, with $c=a_1$ and $\tau=\rho_c^{-1}\sigma$;
the linear coefficient of $\tau(x)$ is $1$.
If $b,b'\in S_n$ and $b'=(\rho_c\tau)\mathbin{\ast}b$, then
\[
 \tau\mathbin{\ast}b=\rho_c^{-1}\mathbin{\ast}b'\in S_n
 \quad\Longrightarrow\quad \tau=\Id.
\]
Thus every element carrying one point of $S_n$ to another is a scaling,
and~\eqref{eq:weighted-scaling-on-divided-powers} gives
\eqref{eq:weighted-parameter-equivalence}.
Conversely, $\rho_c$ realizes every relation
\eqref{eq:weighted-parameter-equivalence}.
\end{proof}

\begin{theorem}[Theorme A]
\label{thm:final-classification}
Every finite-dimensional simple Novikov algebra over $\kk$ is isomorphic
to $\mathcal{N}(\alpha,\beta)$ for some $n\ge0$, $\alpha\in\kk^n$, and
$\beta\in\kk$, with $\beta=1$ if $n=0$.
Two such normal forms are isomorphic if and only if they have the same
dimension and their parameters satisfy
\eqref{eq:weighted-parameter-equivalence} for some $c\in\kk^\times$.
Every isomorphism $\mathcal{N}(\alpha,\beta)\to\mathcal{N}(\alpha',\beta')$ is
\[
 \phi_c=c^{-1}\rho_c,\qquad
 \phi_c(y_i)=c^iy_i\quad(-1\le i\le p^n-2),
\]
for a unique such scalar $c$, and every such scalar gives an isomorphism.
For $n=0$, set $\rho_c=\Id_\kk$.
\end{theorem}

\begin{proof}
For $n=0$, the model is $\kk$ and $\beta=\beta'=1$, so
\eqref{eq:weighted-parameter-equivalence} gives $c=1$.
Suppose $n>0$.  By Proposition~\ref{prop:normalized-differential-data},
every simple Novikov algebra of dimension $p^n$ is isomorphic to
$\OO(1;n)(\delta,b)$.
By Proposition~\ref{prop:residual-orbit}, we may reduce $b$ to
$b_{\alpha,\beta}$.
The divided-power product gives
\begin{align*}
 x^{(i+1)}\delta(x^{(j+1)})
 &=\binom{i+j+1}{j}x^{(i+j+1)},\\
 x^{(i+1)}x^{(j+1)}x^{(p^s-2)}
 &=\binom{i+j+2}{i+1}\binom{i+j+p^s}{p^s-2}
   x^{(i+j+p^s)},\\
 x^{(i+1)}x^{(j+1)}x^{(p^n-1)}
 &=\delta_{i,-1}\delta_{j,-1}x^{(p^n-1)}.
\end{align*}
For $j=-1$, both sides of the first equality are zero by our convention
for binomial coefficients.  In the last equality, every term with $i>-1$ or $j>-1$
has degree at least $p^n$ and vanishes.
With $y_i=x^{(i+1)}$, substituting these identities into
$u\circ v=u(\delta(v)+b_{\alpha,\beta}v)$ gives~\eqref{eq:explicit-algebra-product}.

By Theorem~\ref{thm:change-data} and
Proposition~\ref{prop:residual-orbit}, the isomorphism criterion is
\eqref{eq:weighted-parameter-equivalence}, and every isomorphism between
normal forms is $c^{-1}\rho_c$.
The dimension determines $n$, and $\phi_c(1)=c^{-1}1$ determines $c$.
\end{proof}

\subsection{Further properties}
\label{subsec:classification-consequences}

\begin{corollary}
\label{cor:automorphism-group}
Let $J=\{j\mid1\le j\le n,\ \alpha_j\ne0\}$.  Via $c\mapsto c^{-1}\rho_c$, the
$\kk$-algebra automorphism group of $\mathcal{N}(\alpha,\beta)$ is
\[
 \Aut(\mathcal{N}(\alpha,\beta))\cong
 \begin{cases}
  \{1\},&\beta\ne0,\\
  \kk^\times,&\beta=0,\ J=\varnothing,\\
  \{c\in\kk^\times:c^{p^g-1}=1\},
       &\beta=0,\ J\ne\varnothing,\quad g=\gcd J.
 \end{cases}
\]
In the last case it is cyclic of order $p^g-1$.
\end{corollary}

\begin{proof}
Theorem~\ref{thm:final-classification} identifies the automorphisms with
scalars satisfying $c^{p^j-1}=1$ for $j\in J$ and $c^{p^n}\beta=\beta$.
If $\beta\ne0$, then $(c-1)^{p^n}=0$, so $c=1$.
For $\beta=0$ and $J\ne\varnothing$, use
$\gcd_{j\in J}(p^j-1)=p^{\gcd J}-1$.
Finally, $\phi_c\phi_d=\phi_{cd}$, so this identification respects the
group operation.
\end{proof}

\Needspace{7\baselineskip}
\begin{corollary}
\label{cor:automatic-idempotent}
Every finite-dimensional simple Novikov algebra over $\kk$ contains a nonzero
idempotent.  In $\mathcal{N}(\alpha,\beta)$, one may take $1$ for $n=0$
and $x=x^{(1)}$ for $n>0$.
\end{corollary}

\begin{proof}
The case $n=0$ is immediate.  For $n>0$, one has
$\delta(x)=1$ and $x^2=2x^{(2)}$.
For $1\le j\le n$, the divided-power product gives
\[
 x^{(p^j-2)}x^{(2)}=\binom{p^j}{2}x^{(p^j)}=0,
 \qquad x^{(p^n-1)}x^{(2)}=0.
\]
The first binomial coefficient vanishes because $p>2$;
the second product vanishes by truncation.
Thus $b_{\alpha,\beta}x^2=0$, and
$x\circ x=x\delta(x)+b_{\alpha,\beta}x^2=x$.
\end{proof}

\Needspace{9\baselineskip}
\begin{proposition}
\label{prop:characteristic-polynomial}
Let $n\ge0$, $\alpha,\alpha'\in\kk^n$, and $\beta,\beta'\in\kk$,
with $\beta=\beta'=1$ if $n=0$.
\begin{enumerate}[label=\textup{(\roman*)},leftmargin=2em]
\item
For the element
\[
 s_\alpha=1-\sum_{j=1}^n\alpha_jx^{(p^j-1)}\in\OO(1;n),
\]
the characteristic and minimal polynomials of the left multiplication
operator $\Lop_{s_\alpha}$ on $\mathcal{N}(\alpha,\beta)$ are both
\begin{equation}
 P_{\alpha,\beta}(X)
 =X^{p^n}+\sum_{j=1}^n\alpha_j^{\,p^{n-j}}X^{p^{n-j}}-\beta.
 \label{eq:canonical-characteristic-polynomial}
\end{equation}
\item
The algebras $\mathcal{N}(\alpha,\beta)$ and
$\mathcal{N}(\alpha',\beta')$ are isomorphic
if and only if there is $c\in\kk^\times$ such that
\begin{equation}
 P_{\alpha',\beta'}(X)=c^{p^n}P_{\alpha,\beta}(X/c).
 \label{eq:characteristic-polynomial-scaling}
\end{equation}
\end{enumerate}
\end{proposition}

\begin{proof}
\textup{(i)} For $n=0$, one has $s_\alpha=1$, $\Lop_{s_\alpha}=\Id$, and
$P_{\alpha,1}(X)=X-1$.  Suppose $n>0$, and put $q=p^n$.
The lowest base-$p$ digit of $p^i-1$ is $p-1$, and that of $p^j-2$
is $p-2$.  Since $p>2$, Lucas's congruence gives
\[
 x^{(p^i-1)}x^{(p^j-2)}=0
 \qquad(1\le i,j\le n).
\]
Also $x^{(q-1)}$ annihilates the maximal ideal.  With
$b=b_{\alpha,\beta}$ and $D=\delta+m_b$, we therefore have
\[
 bs_\alpha=b,\qquad
 \delta(s_\alpha)=-\sum_{j=1}^n\alpha_jx^{(p^j-2)},\qquad
 D(s_\alpha)=\beta x^{(q-1)}.
\]
Since $s_\alpha(0)=1$, we may conjugate $\Lop_{s_\alpha}$ by
multiplication with $s_\alpha$.  The identities
$\Lop_{s_\alpha}=m_{s_\alpha}D$ and
$Dm_{s_\alpha}=m_{s_\alpha}\delta+m_{D(s_\alpha)}$ give
\begin{equation}
 m_{s_\alpha}^{-1}\Lop_{s_\alpha}m_{s_\alpha}
 =m_{s_\alpha}\delta+\beta m_{x^{(q-1)}}=:E.
 \label{eq:canonical-operator-conjugation}
\end{equation}

We use the ordered basis $x^{(0)},\ldots,x^{(q-1)}$, with rows and columns
indexed by $0,\ldots,q-1$.
For $1\le r<q$ and $1\le j\le n$, Lucas's congruence gives, in $\kk$,
\[
 \binom{r+p^j-2}{p^j-1}
 =
 \begin{cases}
  1,&p^j\mid r-1,\\
  0,&p^j\nmid r-1.
 \end{cases}
\]
Thus $E(1)=\beta x^{(q-1)}$ and, for $1\le r<q$,
\begin{equation}
 E(x^{(r)})=x^{(r-1)}-
 \sum_{\substack{1\le j\le n,\ p^j\mid r-1\\r+p^j-2<q}}
 \alpha_jx^{(r+p^j-2)}.
 \label{eq:canonical-operator-matrix}
\end{equation}
In particular, the entries immediately above the diagonal are $1$,
and all entries farther above the diagonal vanish.
To compare the minimal and characteristic polynomials, note that
$E(x^{(r)})\equiv x^{(r-1)}\pmod{F^r}$ for $1\le r<q$.
Induction gives
\[
 E^j(x^{(q-1)})\equiv x^{(q-1-j)}\pmod{F^{q-j}}
 \qquad(0\le j<q).
\]
These vectors form a basis, so the minimal polynomial of $E$ has
degree $q$ and therefore equals its characteristic polynomial.
The same holds for $\Lop_{s_\alpha}$ by~\eqref{eq:canonical-operator-conjugation}.

We now compute the characteristic polynomial recursively in $n$.
Let $C_n(X)$ be the characteristic polynomial of $E$ when $\beta=0$, and
set $C_0(X)=X$.
If also $\alpha_n=0$, grouping consecutive basis vectors into $p$
blocks of size $p^{n-1}$ makes the matrix block upper triangular.
A term indexed by $j<n$ contributes an entry in column
$mp^j+1$ and row $(m+1)p^j-1$, so it stays within one such block.  The only entries
joining different blocks come from $\delta$ and lie above the block
diagonal.  Translation by a multiple of $p^{n-1}$ preserves the
divisibility conditions in~\eqref{eq:canonical-operator-matrix}, so each
diagonal block is the matrix for $(\alpha_1,\ldots,\alpha_{n-1})$
in dimension $p^{n-1}$.
The characteristic polynomial is therefore $C_{n-1}(X)^p$.

Restoring $\alpha_n$ changes only the entry in row $q-1$, column $1$
by $-\alpha_n$, while restoring $\beta$ changes only row $q-1$, column $0$
by $\beta$.
Since both changes are in the last row, $\det(X\Id-E)$ is affine linear
in $\alpha_n$ and $\beta$, with cofactors independent of both parameters.
The minor for column $0$ is lower triangular with diagonal entries $-1$.
For column $1$, expansion along column $0$ gives a factor $X$ and a
lower triangular minor with diagonal entries $-1$.  Thus the cofactors are
\[
 \operatorname{cof}_{q-1,0}(X\Id-E)=(-1)^{q-1}(-1)^{q-1}=1,
 \qquad
 \operatorname{cof}_{q-1,1}(X\Id-E)=(-1)^qX(-1)^{q-2}=X.
\]
Thus
\begin{align*}
 C_n(X)&=C_{n-1}(X)^p+\alpha_nX
       =X^{p^n}+\sum_{j=1}^n\alpha_j^{p^{n-j}}X^{p^{n-j}},\\
 \det(X\Id-E)&=C_n(X)-\beta,
\end{align*}
which proves~\eqref{eq:canonical-characteristic-polynomial}.

\textup{(ii)} For two normal forms of dimension $p^n$,
\eqref{eq:characteristic-polynomial-scaling} is equivalent to
\[
 (\alpha'_j)^{p^{n-j}}=c^{p^n-p^{n-j}}\alpha_j^{p^{n-j}},
 \qquad \beta'=c^{p^n}\beta.
\]
Since the Frobenius map $t\mapsto t^p$ on $\kk$ is injective,
the first equalities give
$\alpha'_j=c^{p^j-1}\alpha_j$.
Theorem~\ref{thm:final-classification} now proves~\textup{(ii)}.
\end{proof}

\begin{remark}
For $n>0$ and $a,b\in\kk$, write the family
in~\cite{Osborn1992b,Xu1996} in the form
\begin{equation}
 \mathcal{X}_n(a,b)
 =\OO(1;n)\bigl((1+b z)\delta,a z\bigr),
 \qquad z=x^{(p^n-1)}.
 \label{eq:Xu-family}
\end{equation}
The parameters $a,b$ and the basis $y_i=x^{(i+1)}$ agree with
\cite[(1.5), p.~907]{Xu1996}.
The parameters $a,c$ in~\cite[Theorem~1.3(a), p.~2730]{Osborn1992b}
correspond to our $a,b$.
For $\partial=(1+bz)\delta$ and $h=1-bz=(1+bz)^{-1}$, one has
\[
 h\partial=\delta,\qquad
 \partial(h)+haz
 =-b(1+bz)x^{(p^n-2)}+az
 =-bx^{(p^n-2)}+az.
\]
Theorem~\ref{thm:change-data}, with $\sigma=\Id$, therefore gives
\begin{equation}
 \mathcal{X}_n(a,b)
 \cong \OO(1;n)\bigl(\delta,-b x^{(p^n-2)}+a z\bigr).
 \label{eq:X-normal-form}
\end{equation}
Thus this family corresponds to $\alpha_j=0$ for $1\le j<n$,
$\alpha_n=-b$, and $\beta=a$.
\end{remark}

\begin{remark}
\label{rem:Osborn-reduction}
The family~\eqref{eq:Xu-family} therefore exhausts the simple Novikov
algebras of dimension $p$, but not those of dimension $p^n$ for $n\ge2$.
By Corollary~\ref{cor:automatic-idempotent}, every normal form contains
a nonzero idempotent, so the idempotent hypothesis does not exclude
the missing classes.
For an explicit example, let
\[
 K=\kk[x_2]/(x_2^p),\qquad T=(1+x_2)\partial_2,\qquad
 \partial_*=\partial_1+x_1^{p-1}T.
\]
Consider the $p^2$-dimensional Novikov algebra
$\mathcal{N}_*=\OO(1;2)(\partial_*,0)$.  For $c\in K$, we have
\[
 \partial_*(c)=x_1^{p-1}T(c),\qquad
 \partial_*(cx_1^i)=icx_1^{i-1}\quad(1\le i<p).
\]
In particular, $\partial_*(cx_1)=c$ and
$\partial_*(x_1^2/2)=x_1$, so the right multiplications generate
all multiplication operators on $\OO(1;2)$.
Thus every Novikov ideal is an associative
$\partial_*$-stable ideal.  Since $(p-1)!=-1$, the displayed formulas give
$\partial_*^p(c)=-T(c)$ and $\partial_*^p(x_1)=0$.  Hence
\[
 \partial_*^p=-T,\qquad T^p=T\ne0.
\]
Every such ideal is therefore stable under both $\partial_1$ and $\partial_2$,
and is either $0$ or $\OO(1;2)$.  Hence $\mathcal{N}_*$ is simple.

To distinguish this algebra from $\mathcal X_2(a,b)$, we compare
right identities and right annihilators.
Its right annihilator $\{a:\mathcal N_*\circ a=0\}$ is
$\ker\partial_*=\kk1$.
The right identities of $\mathcal N_*$ are precisely
\[
 \partial_*(e)=1
 \quad\Longleftrightarrow\quad
 e=x_1+c,\qquad c\in\kk.
\]
For $\mathcal X_2(a,b)$, put $z=x^{(p^2-1)}$.
Comparing divided-power coefficients gives
\[
 (1+bz)\delta(e)+aze=1
 \quad\Longleftrightarrow\quad
 e=x+c,\qquad ac+b=0.
\]
An isomorphism preserves right identities.  Since $\mathcal N_*$
has more than one, an isomorphism
$\phi:\mathcal N_*\to\mathcal X_2(a,b)$ would force $a=b=0$.
Both right annihilators are then $\kk1$, so $\phi(1)=d1$ for some
$d\in\kk^\times$, and
\[
 \phi\partial_*\phi^{-1}
 =\phi L_1\phi^{-1}=L_{\phi(1)}=d\delta.
\]
This is impossible, since
\[
 \partial_*^{p^2}=(-T)^p=-T\ne0,
 \qquad (d\delta)^{p^2}=0.
\]
Thus $\mathcal N_*$ is not isomorphic to any $\mathcal X_2(a,b)$.

We can also identify its normal form by computing the characteristic
polynomial of $\Lop_1$.  On each $p$-dimensional invariant subspace
$(1+x_2)^j\kk[x_1]/(x_1^p)$, $0\le j<p$, we have
$\partial_*^p=-j\Id$ and hence $(\partial_*+j\Id)^p=0$.
Its characteristic polynomial is therefore $(X+j)^p=X^p+j$.
Consequently,
\begin{equation}
 \det(X\Id-\Lop_1)=\prod_{j=0}^{p-1}(X^p+j)=X^{p^2}-X^p.
 \label{eq:counterexample-characteristic-polynomial}
\end{equation}
In its normal form, the right annihilator is $\ker D$, where
$D=\delta+m_{b_{\alpha,\beta}}$.  Let $s_\alpha$ be as in
Proposition~\ref{prop:characteristic-polynomial}.
Since $s_\alpha$ is a unit and $P_{\alpha,\beta}(0)=-\beta$, the identity
$\Lop_{s_\alpha}=m_{s_\alpha}D$ shows that $\ker D\ne0$ precisely when $\beta=0$.
In that case $D(s_\alpha u)=s_\alpha\delta(u)$ for every $u\in\OO(1;2)$, so
$\ker D=s_\alpha\ker\delta=\kk s_\alpha$.
An isomorphism $\mathcal N_*\to\mathcal N(\alpha,0)$ therefore sends
$1$ to $cs_\alpha$ for some $c\in\kk^\times$.
Comparing characteristic polynomials gives
\[
 X^{p^2}-X^p=c^{p^2}P_{\alpha,0}(X/c),
 \qquad \alpha_2=0,\quad c^{p-1}\alpha_1=-1.
\]
Theorem~\ref{thm:final-classification} therefore gives
$\mathcal{N}_*\cong\mathcal{N}((-1,0),0)$,
which lies outside the family~\eqref{eq:Xu-family}.

The same example locates the gap in Osborn's reduction.
Taking the right identity $e=x_1$ and $u=1$ in the construction
of~\cite[Lemmas~3.5--3.6, pp.~2740--2742]{Osborn1992b}
gives the auxiliary algebra $\mathcal{A}_0=Kx_1\cong K$, with product
$(cx_1)\circ(dx_1)=cdx_1$ for $c,d\in K$.
On p.~2742, the auxiliary derivation $\partial$ is defined by
\[
 u\circ(u\circ t)=v\circ\partial(t)\quad(t\in\mathcal{A}_0),
 \qquad v=\frac{x_1^{p-1}}{(p-1)!}=-x_1^{p-1}.
\]
Here $u,v$ are denoted by $d,x_{p-2}$ in that proof.
For $c\in K$, we have
\[
 u\circ(u\circ(cx_1))=x_1^{p-1}T(c)
 =v\circ(-T(c)x_1).
\]
Since $t\mapsto v\circ t$ is injective on $\mathcal{A}_0$, this gives
\begin{equation}
 \partial(cx_1)=-T(c)x_1\quad(c\in K),\qquad
 \partial^p=\partial\ne0.
 \label{eq:Osborn-nonnilpotent-derivation}
\end{equation}
Since $T$ is a unit multiple of $\partial_2$, the algebra
$\mathcal{A}_0$ is $\partial$-simple, and $\ker\partial=\kk e$.
After proving $\partial$-simplicity, the argument uses a basis
$\{t^{(k)}:0\le k<p^r\}$ satisfying
\[
 t^{(0)}=e,\qquad
 \partial(t^{(k)})=t^{(k-1)}\quad(1\le k<p^r).
\]
Here $\dim\mathcal{A}_0=p$, so $r=1$ and this would force
$\partial^p=0$, contrary to~\eqref{eq:Osborn-nonnilpotent-derivation}.
Thus $\partial$-simplicity does not justify this choice of basis.
The use of~\cite[Lemma~3.6]{Osborn1992b} immediately before
\cite[Theorem~2.9, p.~914]{Xu1996} brings the same gap into the latter reduction.
\end{remark}

\subsection{Novikov--Poisson algebras with simple Novikov part}
\label{subsec:novikov-poisson-classification}

\begin{definition}[\cite{Xu1996,Xu1997}]
A \emph{Novikov--Poisson algebra} is a vector space $\mathcal A$ with a
commutative associative product $\cdot$ and a Novikov product $\circ$
satisfying
\begin{align*}
 (x\cdot y)\circ z&=x\cdot(y\circ z),\\
 (x\circ y)\cdot z-x\circ(y\cdot z)
 &=(y\circ x)\cdot z-y\circ(x\cdot z)
\end{align*}
for all $x,y,z\in\mathcal A$.
\end{definition}

\begin{corollary}
\label{cor:novikov-poisson-classification}
Every finite-dimensional Novikov--Poisson algebra whose Novikov part
is simple is isomorphic to $\OO(1;n)$ with products
\begin{equation}
 f\circ g=f\bigl(\delta(g)+b_{\alpha,\beta}g\bigr),
 \qquad f\cdot g=u_\gamma fg,
 \qquad u_\gamma=\sum_{r=0}^{p^n-1}\gamma_rx^{(r)},
 \label{eq:novikov-poisson-normal-form}
\end{equation}
where $(n,\alpha,\beta)$ are as in
Theorem~\ref{thm:final-classification},
$\gamma\in\kk^{p^n}$ is arbitrary, and juxtaposition denotes the
divided-power product.
Two such algebras are isomorphic if and only if they have the same $n$
and there exists $c\in\kk^\times$ such that
\begin{equation}
 \alpha'_j=c^{p^j-1}\alpha_j,\qquad
 \beta'=c^{p^n}\beta,\qquad
 \gamma'_r=c^{r+1}\gamma_r
 \label{eq:novikov-poisson-isomorphism}
\end{equation}
for $1\le j\le n$ and $0\le r<p^n$.
\end{corollary}

\begin{proof}
By Theorem~\ref{thm:final-classification}, we may identify the Novikov
part with $\mathcal N(\alpha,\beta)$.  Put $M_f(g)=f\cdot g$.
The first compatibility identity gives $[M_f,R_g]=0$.
By Lemma~\ref{lem:normalized-simple}, the right multiplications generate
all multiplication operators of $\OO(1;n)$.  Hence each $M_f$
commutes with these operators, and commutativity gives
\[
 f\cdot g=gM_f(1)=gM_1(f)=fgM_1(1).
\]
Thus $f\cdot g=ufg$ for a unique $u\in\OO(1;n)$.
Conversely, every such product is commutative and associative and
satisfies the first compatibility identity.  The second follows from
\[
 (f\circ g)\cdot h-f\circ(g\cdot h)=-fg\,\delta(uh).
\]
Finally, by Theorem~\ref{thm:final-classification}, all isomorphisms
between the Novikov parts are $\phi_c=c^{-1}\rho_c$, where
$\rho_c(x^{(r)})=c^rx^{(r)}$.
Preservation of $\cdot$ is equivalent to
$u_{\gamma'}=c\rho_c(u_\gamma)$, which gives
\eqref{eq:novikov-poisson-isomorphism}.
\end{proof}

\Needspace{12\baselineskip}
\section{Classification of irreducible bimodules}
\label{sec:bimodules}

For the algebras in Theorem~A, we recover an action of the underlying
associative algebra on every nontrivial irreducible bimodule and prove
that each such bimodule is finite-dimensional.
We then classify these bimodules and determine their isomorphisms.

\begin{definition}
Let $\mathcal{N}$ be a Novikov algebra.
\begin{enumerate}[label=\textup{(\roman*)},leftmargin=2em]
\item
A \emph{bimodule} over $\mathcal{N}$ is a vector space $V$ with linear
maps $\ell,r:\mathcal{N}\to\End_\kk(V)$ such that the split null extension
$\mathcal{N}\oplus V$, with product
\[
 (x+u)\circ(y+v)=x\circ y+\ell(x)v+r(y)u,
\]
is a Novikov algebra.  Equivalently, for all $x,y\in\mathcal{N}$,
\begin{align}
 \ell(x\circ y)&=r(y)\ell(x),
 \label{eq:bimodule-RC}\\
 [r(x),r(y)]&=0,
 \label{eq:bimodule-right-commuting}\\
 [\ell(x),\ell(y)]&=\ell(x\circ y-y\circ x),
 \label{eq:bimodule-left-Lie}\\
 [\ell(x),r(y)]&=r(x\circ y)-r(y)r(x).
 \label{eq:bimodule-mixed}
\end{align}
\item
A nonzero bimodule is \emph{irreducible} if it has no nonzero proper
subspace stable under both actions.
\item
A bimodule is \emph{nontrivial} if its two actions are not both zero.
\item
The \emph{regular bimodule} is $\mathcal{N}$ with
$\ell(x)=L_x$ and $r(x)=R_x$.
\end{enumerate}
\end{definition}

The only trivial irreducible bimodule, up to isomorphism, is the
one-dimensional bimodule with both actions zero.
We do not assume that bimodules are finite-dimensional.

Fix $n\ge0$ and write
\[
 \mathcal{A}=\OO(1;n),\qquad D=\delta+m_b,\qquad
 \mathcal{N}(b)=\mathcal{A}(\delta,b)\quad(b\in \mathcal{A}).
\]
When $n=0$, take $b=1$.
For the normal forms in Theorem~A, we have
$\mathcal{N}(\alpha,\beta)=\mathcal{N}(b_{\alpha,\beta})$.
By Lemma~\ref{lem:normalized-simple}, $\mathcal{N}(b)$ is simple and
its right multiplications generate all multiplication operators
on $\mathcal{A}$.

\begin{lemma}
\label{lem:bimodule-reconstruction}
Let $V$ be a nontrivial irreducible $\mathcal{N}(b)$-bimodule and put
$T=\ell(1)$.  There is a unique injective unital algebra homomorphism
$\pi:\mathcal{A}\to\End_\kk(V)$ such that
\begin{align}
 \ell(x)&=\pi(x)T,& r(x)&=\pi(D(x)),
 \label{eq:module-reconstructed-actions}\\
 [T,\pi(a)]&=\pi(\delta(a))&& (a\in\mathcal{A}).
 \label{eq:module-derivation-relation}
\end{align}
A subspace of $V$ is a subbimodule if and only if it is stable under
$\pi(\mathcal{A})$ and $T$.
\end{lemma}

\begin{proof}
For $n=0$, put $L=\ell(1)$ and $R=r(1)$.
The bimodule identities give $RL=L$ and $[L,R]=R-R^2$.
The subspace $\ker(R-\Id)$ is stable under $L$ and $R$ and is nonzero:
it contains $\operatorname{im}L$ if $L\ne0$, and
$\operatorname{im}R\ne0$ if $L=0$, since then $R^2=R$.
Irreducibility gives $R=\Id$, so take $\pi(a)=a\Id$.

Suppose now that $n>0$.
To recover an associative action on $V$, we first show that the images
of the left actions span $V$.  Let
$U=\sum_{x\in\mathcal{A}}\ell(x)V$.  The bimodule identities give
\[
 r(y)\ell(x)=\ell(x\circ y),\qquad
 \ell(y)\ell(x)=\ell(x)\ell(y)+\ell(y\circ x-x\circ y),
\]
so $U$ is a subbimodule.
If $U=0$, then $\ell=0$, and
\eqref{eq:bimodule-right-commuting} and~\eqref{eq:bimodule-mixed}
make $r$ a nonzero homomorphism into a commutative associative algebra.
Simplicity makes $r$ injective, contradicting
$1\circ x^{(1)}-x^{(1)}\circ1=1$.
Thus $U=V$.

We can now transfer polynomial relations from the regular right
multiplications to the right actions on $V$.
As in~\cite{Osborn1995modules} and~\cite[proof of Lemma~4.1]{Xu2001},
repeated use of $r(u)\ell(x)=\ell(R_ux)$ gives, for
$u_1,\ldots,u_m\in\mathcal{A}$ and $P\in\kk[t_1,\ldots,t_m]$,
\begin{equation}
 P(r(u_1),\ldots,r(u_m))\ell(x)
 =\ell\bigl(P(R_{u_1},\ldots,R_{u_m})x\bigr).
 \label{eq:right-relation-transfer}
\end{equation}
If $P(R_{u_1},\ldots,R_{u_m})=0$, then
$P(r(u_1),\ldots,r(u_m))$ annihilates $U=V$.
Since $R_u=m_{D(u)}$, we have
\[
 P(R_{u_1},\ldots,R_{u_m})
 =m_{P(D(u_1),\ldots,D(u_m))}.
\]
It follows that every polynomial relation among the $D(u)$ holds
among the $r(u)$.
Since $D(\mathcal{A})$ generates $\mathcal{A}$ by
Lemma~\ref{lem:normalized-simple}, the assignment $D(y)\mapsto r(y)$
extends uniquely to a unital algebra homomorphism
$\pi:\mathcal{A}\to\End_\kk(V)$.
Explicitly,
\[
 \pi\bigl(P(D(u_1),\ldots,D(u_m))\bigr)
 =P(r(u_1),\ldots,r(u_m)),\qquad \pi(1)=\Id_V.
\]
Equation~\eqref{eq:right-relation-transfer} also gives
$\pi(a)\ell(x)=\ell(ax)$.  Taking $x=1$ yields $\ell(a)=\pi(a)T$.

The remaining commutator relation follows from
\eqref{eq:bimodule-left-Lie}, which gives
$[T,\ell(x)]=\ell(\delta(x))$.
Using $\pi(a)\ell(x)=\ell(ax)$, we obtain
\[
 [T,\pi(a)]\ell(x)
 =\ell(\delta(ax))-\ell(a\delta(x))
 =\pi(\delta(a))\ell(x).
\]
Since $U=V$, this proves~\eqref{eq:module-derivation-relation}.
This relation also proves faithfulness: $\ker\pi$ is a proper
$\delta$-stable ideal and is zero by
$\delta$-simplicity.
The final assertion follows because $\pi(\mathcal{A})$ is generated
by the right actions and $\ell(x)=\pi(x)T$.
\end{proof}

\begin{lemma}
\label{lem:automatic-finite-dimensionality}
Every nontrivial irreducible $\mathcal{N}(b)$-bimodule is finite-dimensional.
\end{lemma}

\begin{proof}
Let $V$ be such a bimodule, with $\pi$ and $T=\ell(1)$ as in
Lemma~\ref{lem:bimodule-reconstruction}.  Write $av=\pi(a)v$.
Let $\mathfrak m$ be the maximal ideal of $\mathcal{A}=\OO(1;n)$.
We first use its nilpotence to choose a vector that will generate $V$
under $T$.  Take the largest $s\ge0$ such that $\mathfrak m^sV\ne0$
and choose $0\ne v\in\mathfrak m^sV$.
Then $\mathfrak m v=0$, so $av=a(0)v$ for every $a\in \mathcal{A}$.
Induction using $aT=Ta-\delta(a)$ gives, for $r\ge0$,
\[
 aT^rv
 =\sum_{j=0}^r(-1)^j\binom rj T^{r-j}\delta^j(a)v
 =\sum_{j=0}^r(-1)^j\binom rj\delta^j(a)(0)T^{r-j}v.
\]
The subspace $\kk[T]v$ is nonzero and stable under $\mathcal{A}$ and $T$,
so it is a subbimodule and $V=\kk[T]v$ by irreducibility.

Put $q=p^n$.  If $V$ were infinite-dimensional, the vectors $T^rv$ would be linearly
independent: any polynomial relation would express all sufficiently high
powers in terms of finitely many lower ones.
In that case
\[
 T^qV=\operatorname{span}_{\kk}\{T^rv:r\ge q\}
\]
would be nonzero and would not contain $v$.
To contradict irreducibility, it remains to check stability under
$\mathcal A$.
Repeated use of~\eqref{eq:module-derivation-relation} gives
\[
 T^q\pi(a)=\sum_{j=0}^q\binom qj\pi(\delta^j(a))T^{q-j}
          =\pi(a)T^q,
\]
since $\binom qj=0$ in $\kk$ for $0<j<q$ and $\delta^q=0$.
Hence $T^qV$ is stable under $\mathcal{A}$ and $T$.  It is a nonzero proper
subbimodule, contradicting irreducibility.
\end{proof}

For $\lambda\in\kk$, let $\mathcal{M}(b,\lambda)$ be the vector space
$\mathcal{A}$ with linear operators
\begin{align}
 \ell_\lambda(x)v
 &=x\delta(v)+\lambda xv,
 \label{eq:canonical-module-left}\\
 r_b(x)v
 &=v\delta(x)+bxv.
 \label{eq:canonical-module-right}
\end{align}
For $b=b_{\alpha,\beta}$, write
$\mathcal{M}(\alpha,\beta,\lambda)=\mathcal{M}(b_{\alpha,\beta},\lambda)$.
In the basis $v_i=x^{(i+1)}$, $-1\le i\le p^n-2$, these actions
give~\eqref{eq:explicit-module-left}--\eqref{eq:explicit-module-right}.

\Needspace{13\baselineskip}
\begin{theorem}[Thoerem B]
\label{thm:irreducible-bimodule-classification}
Fix $\mathcal{N}(\alpha,\beta)$ as in Theorem~A.
\begin{enumerate}[label=\textup{(\roman*)},leftmargin=2em]
\item
For each $\lambda\in\kk$, the bimodule
$\mathcal{M}(\alpha,\beta,\lambda)$ is nontrivial and irreducible
of dimension $p^n$.  Every nontrivial irreducible
$\mathcal{N}(\alpha,\beta)$-bimodule is isomorphic to one of these.

\item
The $\mathcal{N}(\alpha,\beta)$-bimodules
$\mathcal{M}(\alpha,\beta,\lambda)$ and
$\mathcal{M}(\alpha,\beta,\lambda')$ are isomorphic if and only if
$\lambda=\lambda'$.
\end{enumerate}
\end{theorem}

\begin{proof}
We prove the assertions for $\mathcal{N}(b)$ with arbitrary
$b\in\mathcal{A}$, taking $b=1$ if $n=0$.
Fix $\lambda\in\kk$ and write $E=\delta+\lambda\Id$.
On $\mathcal{M}(b,\lambda)=\mathcal{A}$, the operators are
$\ell_\lambda(x)=m_xE$ and $r_b(x)=m_{D(x)}$.
Since $x\circ y=xD(y)$, one has
\[
 \ell_\lambda(x\circ y)=m_{xD(y)}E=r_b(y)\ell_\lambda(x),
 \qquad [r_b(x),r_b(y)]=0.
\]
Identities~\eqref{eq:bimodule-left-Lie} and~\eqref{eq:bimodule-mixed}
follow from
\begin{align*}
 [m_xE,m_yE]
 &=m_{x\delta(y)-y\delta(x)}E
  =\ell_\lambda(x\circ y-y\circ x),\\
 [m_xE,m_{D(y)}]
 &=m_{x\delta(D(y))}\\
 &=m_{D(xD(y))-D(y)D(x)}
  =r_b(x\circ y)-r_b(y)r_b(x).
\end{align*}
Thus $\mathcal{M}(b,\lambda)$ is a bimodule.
Its right actions generate all multiplication operators on $\mathcal{A}$,
so it is nontrivial and every subbimodule is an ideal of $\mathcal{A}$.
Such an ideal is stable under $\ell_\lambda(1)=\delta+\lambda\Id$ and hence
under $\delta$.  The $\delta$-simplicity of $\mathcal{A}$ proves irreducibility.

Conversely, let $V$ be a nontrivial irreducible $\mathcal{N}(b)$-bimodule, with
$\pi$ and $T=\ell(1)$ as in Lemma~\ref{lem:bimodule-reconstruction}.
Write $av=\pi(a)v$.
Since $V$ is finite-dimensional by
Lemma~\ref{lem:automatic-finite-dimensionality} and $\kk$ is
algebraically closed, we may choose $0\ne w\in V$ with
$Tw=\lambda w$ for some $\lambda\in\kk$, following the eigenvector
method in~\cite[Proposition~3.1]{Xu1996} and
\cite[proof of Theorem~4.3, Case~1]{Xu2001}.
For $a\in \mathcal{A}$,
\[
 T(aw)=\bigl(\delta(a)+\lambda a\bigr)w.
\]
Thus $\mathcal{A}w$ is a nonzero subspace stable under $\mathcal{A}$ and $T$, so $V=\mathcal{A}w$
by irreducibility.  The map $\mathcal{A}\to V$, $a\mapsto aw$, is also
injective: its kernel is an ideal of $\mathcal{A}$, and if $aw=0$,
the displayed identity gives $\delta(a)w=0$.
The kernel is therefore $\delta$-stable and proper, so it is zero by
$\delta$-simplicity.
The resulting linear isomorphism $F:\mathcal{A}\to V$, $F(a)=aw$, intertwines
both actions: for $x,a\in \mathcal{A}$,
\begin{align*}
 \ell(x)F(a)&=xT(aw)
             =F\bigl(x\delta(a)+\lambda xa\bigr),\\
 r(x)F(a)&=(\delta(x)+bx)aw=F\bigl((\delta(x)+bx)a\bigr).
\end{align*}
Hence $V\cong\mathcal{M}(b,\lambda)$ and $\dim V=p^n$.

On $\mathcal{M}(b,\lambda)$ one has
\begin{equation}
 \ell(1)^{p^n}=(\delta+\lambda\Id)^{p^n}=\lambda^{p^n}\Id.
 \label{eq:bimodule-parameter-power}
\end{equation}
Thus $\lambda$ is the unique eigenvalue of $\ell(1)$.
If $\mathcal{M}(b,\lambda)\cong \mathcal{M}(b,\lambda')$, then
$\lambda^{p^n}=(\lambda')^{p^n}$ and hence $\lambda=\lambda'$.
\end{proof}

For bimodules $V$ over $\mathcal{N}$ and $V'$ over $\mathcal{N}'$,
an isomorphism over an algebra isomorphism
$\phi:\mathcal{N}\to\mathcal{N}'$ is a linear bijection $F:V\to V'$
satisfying $F(x\circ v)=\phi(x)\circ F(v)$ and
$F(v\circ x)=F(v)\circ\phi(x)$.
For $\phi=\Id_{\mathcal{N}}$, this is the usual notion of an isomorphism
of $\mathcal{N}$-bimodules.

\Needspace{9\baselineskip}
\begin{corollary}
\label{cor:module-parameter-change}
Let $n>0$.
\begin{enumerate}[label=\textup{(\roman*)},leftmargin=2em]
\item
Let $(\sigma,h)\in G_n$, let $a_1$ be the coefficient of $x$ in
$\sigma(x)$, and put $b'=(\sigma,h)\ast b$.
Then $\mathcal{M}(b,\lambda)$ and $\mathcal{M}(b',a_1\lambda)$ are
isomorphic over $\phi=h^{-1}\sigma$.
In particular, $G_n^1$ preserves the module parameter.
\item
Let $\phi_c=c^{-1}\rho_c$ be an algebra isomorphism between
$\mathcal{N}(\alpha,\beta)$ and $\mathcal{N}(\alpha',\beta')$
as in Theorem~\ref{thm:final-classification}.
The bimodules $\mathcal{M}(\alpha,\beta,\lambda)$ and
$\mathcal{M}(\alpha',\beta',\lambda')$ are isomorphic over $\phi_c$
if and only if $\lambda'=c\lambda$.
Every such isomorphism is a nonzero scalar multiple of $\rho_c$, where
\[
 \rho_c(v_i)=c^{i+1}v_i\qquad(-1\le i\le p^n-2).
\]
\end{enumerate}
\end{corollary}

\begin{proof}
\textup{(i)} On $\mathcal{M}(b,\lambda)$, define the
$\mathcal{N}(b')$-actions by
\[
 \ell'(a)=\ell(\phi^{-1}(a)),\qquad
 r'(a)=r(\phi^{-1}(a)).
\]
Put $q=p^n$ and $g=\sigma^{-1}(h)$.
Since $\phi^{-1}(1)=g$, the left action of $1$ becomes
$T'=m_g(\delta+\lambda\Id)$.
By the classification for arbitrary $b$ established in the proof of
Theorem~\ref{thm:irreducible-bimodule-classification},
$T'$ has a unique eigenvalue $\mu$, and
the determinant therefore determines the new parameter:
\[
 \mu^q=\det T'=g(0)^q\lambda^q=a_1^q\lambda^q.
\]
Here $g(0)=h(0)=a_1$, and multiplication by $g$ is triangular
with diagonal entries $g(0)$.
Hence $\mu=a_1\lambda$.

\textup{(ii)} Taking $\sigma=\rho_c$ and $h=c$ in~\textup{(i)}
gives the necessary condition $\lambda'=c\lambda$.
Writing $b=b_{\alpha,\beta}$ and $b'=b_{\alpha',\beta'}$,
we have $b'=c\rho_c(b)$.
Together with $\delta\rho_c=c\rho_c\delta$, this gives, for $x,v\in\mathcal A$,
\begin{align*}
 \ell_{c\lambda}(\phi_c(x))\rho_c(v)
 &=c^{-1}\rho_c(x)\bigl(\delta\rho_c(v)+c\lambda\rho_c(v)\bigr)\\
 &=\rho_c\bigl(x\delta(v)+\lambda xv\bigr),\\
 r_{b'}(\phi_c(x))\rho_c(v)
 &=\bigl(\rho_c(\delta(x))+\rho_c(bx)\bigr)\rho_c(v)\\
 &=\rho_c\bigl((\delta(x)+bx)v\bigr).
\end{align*}
Thus $\rho_c$ intertwines both actions over $\phi_c$.
Composing any such isomorphism with $\rho_c^{-1}$ on the left gives
a bimodule automorphism $F$ of $\mathcal M(b,\lambda)$.
Since the right actions generate all multiplication operators, $F$
commutes with these operators.  Put $f=F(1)$.  Then
\[
 F(a)=F(m_a1)=aF(1)=af,\qquad
 0=[\delta+\lambda\Id,m_f]=m_{\delta(f)}.
\]
Hence $f\in\ker\delta=\kk1$, so $F$ is a nonzero scalar.
Thus all isomorphisms over $\phi_c$ are nonzero scalar multiples of
$\rho_c$.
\end{proof}

\begin{remark}
\label{rem:PZ-bimodules}
Let $n>0$ and $\mathcal{A}=\OO(1;n)$.
The construction of Pozhidaev and
Zhelyabin~\cite[Proposition~1 and Theorems~4, 6]{PZmodules}
gives all nontrivial finite-dimensional irreducible
$\mathcal{N}(b)$-bimodules, up to isomorphism, on $\mathcal{A}$ with
right action~\eqref{eq:canonical-module-right} and left action
\[
 \widetilde\ell_\gamma(x)v=x\delta(v)+(b+\gamma)xv,
 \qquad \gamma\in\kk.
\]
Theorem~B removes the finite-dimensionality assumption.
To compare the parameters, consider the regular bimodule.
The classification established in the proof of
Theorem~\ref{thm:irreducible-bimodule-classification} shows that
$D=\delta+m_b$ has a unique eigenvalue $\kappa$.
Choose $0\ne h\in\mathcal{A}$ with $Dh=\kappa h$.
Since $\delta(h)=(\kappa-b)h$, the ideal $\mathcal{A}h$ is
$\delta$-stable, so $h$ is a unit.
We have
\[
 (\delta+m_{b+\gamma})m_h
 =m_h\bigl(\delta+(\gamma+\kappa)\Id\bigr).
\]
Since $m_h$ commutes with every multiplication operator, it identifies
$\mathcal{M}(b,\gamma+\kappa)$ with this bimodule.  Thus $\lambda=\gamma+\kappa$.
For $b=b_{\alpha,\beta}$, the element $s_\alpha$ in
Proposition~\ref{prop:characteristic-polynomial} satisfies
$\Lop_{s_\alpha}=m_{s_\alpha}D$ and $s_\alpha(0)=1$.
Since $p^n$ is odd, we obtain
\[
 \kappa^{p^n}=\det D=\det\Lop_{s_\alpha}=\beta,
 \qquad \lambda=\gamma+\beta^{1/p^n}.
\]
In~\cite{PZmodules}, the associative action is recovered through a
Novikov--Poisson structure on the split null extension.  We use polynomial
relations among right multiplication operators
(Lemma~\ref{lem:bimodule-reconstruction}).
\end{remark}

\begin{remark}
Fix $n>0$ and $a,b\in\kk$.
Put $z=x^{(p^n-1)}$, $u=1+bz$, and $\partial=u\delta$.
For $\mu\in\kk$, consider the actions on $\OO(1;n)$
\[
 \ell_\mu(x)=m_x(\partial+\mu\Id),\qquad
 r(x)=m_{\partial(x)+azx}.
\]
As in the proof of
Theorem~\ref{thm:irreducible-bimodule-classification},
these define a nontrivial irreducible
$\mathcal{X}_n(a,b)$-bimodule.
Under the normalization~\eqref{eq:X-normal-form}, the left action
of $1$ becomes
\[
 T=\ell_\mu(u^{-1})=\delta+m_{\mu-b\mu z}.
\]
In the divided-power basis, $T$ has diagonal entries $\mu$,
entries $1$ immediately above the diagonal, and the additional
entry $-b\mu$ in the bottom-left corner.  Since $p^n$ is odd,
Theorem~\ref{thm:irreducible-bimodule-classification} gives
\begin{equation}
 \lambda^{p^n}=\det T=\mu^{p^n}-b\mu.
 \label{eq:Xu-to-canonical-module-parameter}
\end{equation}
Every $\lambda\in\kk$ occurs, so
Theorem~\ref{thm:irreducible-bimodule-classification} shows that
these exhaust the nontrivial irreducible
$\mathcal{X}_n(a,b)$-bimodules.
For fixed $a,b$, two values $\mu,\mu'$ give isomorphic bimodules
if and only if
\[
 (\mu-\mu')^{p^n}=b(\mu-\mu').
\]
If $b\ne0$, each isomorphism class is represented by exactly $p^n$
values of $\mu$, since the polynomial $X^{p^n}-bX$ has derivative $-b$.
If $b=0$, then $\lambda=\mu$.
Our parameter $\mu$ is denoted by $\lambda$ in
\cite[Main Theorem~2, p.~908]{Xu1996}.
For $b\ne0$, the value $\mu=0$ is equivalent to any
$\eta$ satisfying $\eta^{p^n-1}=b$, so the restriction $\mu\ne0$ in
\cite[Main Theorem~2]{Xu1996} does not remove an isomorphism class.
\end{remark}

\section*{\texorpdfstring{\normalfont\bfseries Acknowledgments}{Acknowledgments}}
The authors thank A.~P. Pozhidaev for helpful discussions and for sharing
the manuscript~\cite{PZmodules}.

\medskip
\noindent\textit{Note added.}
Upon completion of this paper, the authors learned from A.~P. Pozhidaev
that he and his collaborators had also obtained a classification and
an isomorphism criterion for finite-dimensional simple Novikov algebras
in an unpublished manuscript.

\end{document}